\documentclass[a4paper, 12pt]{article}
\usepackage{anysize}
\marginsize{3,5cm}{2,5cm}{2,5cm}{2,5cm}
\usepackage{amssymb}
\usepackage{amsmath,amsbsy,amssymb,amscd}
\usepackage{t1enc}
\usepackage[cp1250]{inputenc}
\usepackage[british]{babel}
\usepackage[all]{xy}
\usepackage{color}\usepackage{amsfonts}
\usepackage{latexsym}
\usepackage{amsthm}
\usepackage{mathrsfs}
\usepackage{hyperref}
\usepackage{graphicx}

\usepackage{color}
\usepackage{caption}
\usepackage{subcaption}
\usepackage{enumerate}
\usepackage{indentfirst}

\usepackage{mathrsfs} 

\DeclareMathAlphabet{\mathpzc}{OT1}{pzc}{L}{it} 

\newtheorem{definition}{Definition}[section]

\newtheorem{proposition}[definition]{Proposition}
\newtheorem{theorem}[definition]{Theorem}
\newtheorem{problem}[definition]{Problem}
\newtheorem{corollary}[definition]{Corollary}
\newtheorem{remark}[definition]{Remark}
\newtheorem{lemma}[definition]{Lemma}

\def\R{\mathbb{R}}
\def\T{\mathbb{T}}

\def\Z{\mathbb{Z}}
\def\N{\mathbb{N}}

\def\Q{\mathbb{Q}}

\def\cD{\mathcal{D}}

\newcommand{\bea}{\begin{eqnarray}}
  \newcommand{\eea}{\end{eqnarray}}
  \newcommand{\beab}{\begin{eqnarray*}}
  \newcommand{\eeab}{\end{eqnarray*}}

  \newcommand{\be}{\begin{equation}}
  \newcommand{\ee}{\end{equation}}

\title{Prime orbits for some smooth flows on $\T^2$}
\author{Adam Kanigowski}
\begin{document}
\baselineskip=14pt \maketitle

\begin{abstract}
 We consider a class of smooth mixing flows $T^{\alpha,\gamma}$ on $\T^2$ with one degenerated fixed point $x_0\in \T^2$ of power type $\gamma\in (-1,0)$ (Kochergin flows). We prove that for a $G_\delta$ dense set of $\alpha\in \T$, a prime number theorem for $T^{\alpha,\gamma}$ holds along a full upper density subsequence. In particular it follows that for every $x\in \T^2\setminus\{x_0\}$, the prime orbit $\{T^{\alpha,\gamma}_{\pm p}x\}_{p \text{ prime}}$ is dense in $\T^2$. 

We also show that  there exists a class of smooth weakly mixing flows on $\T^2$ for which a  prime number theorem holds. In fact we show that there exists a dense set of smooth functions (in the uniform topology) for which prime number theorem holds quantitatively (with an error term $\log^{-A}N$).
\end{abstract}
\section{Introduction}
We study distribution of prime orbits in the class of smooth flows on $\T^2$. Our main focus are so called {\em Kochergin flows} on $\T^2$, i.e. smooth flows on $\T^2$ with at least one (degenerated) fixed point. As shown by Kochergin \cite{Koc1}, such flows are always mixing. Kochergin flows belong to the broader class of {\em multivalued Hamiltonian flows} on surfaces. The latter class is a representative of so called {\em parabolic dynamical systems}. Ergodic and mixing properties of multivalued Hamiltonian flows have been studied since 1970's and they are by now well understood \cite{Koc1,Koc3,Ar,KS,fay,FK,Ulc1,Ulc2,CW,Rav}. Multivalued hamiltonian flows always have a finite (non-empty) set of singularities (fixed points) and the dynamics depends strongly on the nature of the singularities. In this paper we will focus on the simplest case, when the surface is the two-dimensional torus and there is only one fixed point (singularity) of {\em power type} (Kochergin flows). In this case the orbits (except the singularity) look qualitatively the same as the orbits of a linear flow on $\T^2$. However the presence of the fixed point, which acts as a stopping point, changes ergodic properties of the flow drastically. Every Kochergin flow has a unique irrational frequency $\alpha$ associated to it and by a result of Kochergin, \cite{Koc1}, for \textbf{any} (irrational) $\alpha$ the corresponding Kochergin flow is mixing. Moreover, under a diophantine condition on $\alpha$, mixing is quantitative (polynomial)  \cite{fay} and in fact for sufficiently strong singularity the spectrum is countable Lebesgue \cite{FFK}. In particular the dynamics of Kochergin flows has similar features to the dynamics of horocycle flows and is therefore characteristic for parabolic dynamics.

We will be also interested in the case where there are no fixed points for the flow. In this case, \cite{COFOSI}, the setting is reduced to {\em reparametrizations} of linear flows on $\T^2$. Such flows belong to the class of {\em elliptic} dynamical systems. Indeed, they are always {\em rigid} \cite{KATOK}, and hence never mixing. Moreover, if the rotation vector is {\em diophantine}, they are smoothly conjugated to the original linear flow (the reparametrization is trivial), \cite{KOL}. On the other hand non-trivial dynamical behavior is possible for {\em liouvillean} reparametrizations. As shown by Sklover, \cite{Shk} (see also \cite{FAY}), such flows might be weakly mixing. One of the main focus of the paper are weakly mixing liouvillean reparametrizations of linear flows.

We will be interested in distribution of prime orbits for Kochergin flows (with one singularity) and for weakly mixing reparametrizations of linear flows. We emphasize that we will study prime orbits of {\em every} point; the behavior of prime orbits for a.e. points is understood  for many (density zero) subsets of $\Z$ and for general dynamical systems, \cite{BO}, \cite{WI}. More precisely, for a (topological) dynamical system $(X,T)$, $\psi\in C(X)$ and $x\in X$ we study expressions of the form ($p$ always denotes prime numbers),
 $$
 \frac{1}{N}\sum_{p\leq N}\psi(T^px)\log p.
 $$
 Following \cite{KLR}, we say that $(X,T)$ satisfies a {\em prime number theorem} if for every $\psi\in C(X)$ and $x\in X$ the above expression has a limit as $N\to +\infty$. We will say that $(X,T)$ satisfies a {\em PNT along a subsequence} if for every $x\in X$ there exists $(N_k)$ such that the above limit exists for every $\psi$ as $N_k\to +\infty$. Recall that PNT is known for  restricted classes of dynamical systems. By the work of Vinogradov \cite{VI} it follows that irrational rotations satisfy a PNT. This has been generalized by Green-Tao \cite{GT} to nilsystems (recall that nilsystems always have a non-trivial Kronecker factor). Moreover, there are classes of symbolic systems for which a PNT holds \cite{Bour, MR,MU,Gr,FM}. Recently, \cite{KLR}, the authors showed that uniquely ergodic analytic Anzai skew products satisfy a PNT. We emphasize that there is no example of a smooth weakly mixing dynamical system for which a PNT holds (all examples above are either symbolic or have a non-trivial Kronecker factor).   One of the main reasons for which it is hard to establish a PNT for (weakly) mixing systems is the following: the main method of establishing a PNT is studying so called type I and type II sums, \cite{VI, GT, SU}. In particular this requires a strong quantitative understanding of {\em joinings} which is usually beyond reach for (weakly) mixing systems (see eg. \cite{SU}). A method recently developed in \cite{KLR} does not use type I and type II sums, but instead relies on strong approximation of the system by a sequence of periodic systems, which is however usually not possible for mixing systems.

 We now pass to the description of our main result which deals with Kochergin flows. As mentioned above, Kochergin flows belong to the class of parabolic dynamical systems in which  one usually has good equidistribution properties at all points (when looking at full orbits of the $\Z$ or $\R$ actions). However, not much is known when one samples the orbits at sparse (density zero) subsets of $\Z$ (or $\R$) even for the most classical examples, i.e. {\em horocycle flows}.  Recall that  Shah conjectured, \cite{Sh}, that the orbits of horocycle flows equdistribute when sampled at polynomial times, i.e. $\{n^\beta\}_{n\in \N}$. A special case of Shah's conjecture, in particular with $\beta<\beta_0<2$, was proven in \cite{Ven} (see also \cite{FFT}). Even less is known when one wants to study the orbits sampled at prime times (which is the main interest of this paper). In \cite{SU} the authors studied the orbits of horocycle flows at prime times and showed that prime orbits visit every set of diameter $>1/10$). However, it is an open question, whether prime orbits of horocycle flows are equidistributed. In fact, it is not known if the orbits are dense. To the best of our knowledge the above short list summarizes all results on distribution of sparse orbits for parabolic dynamical systems. In fact, it is not even known if there exists a mixing system (even weakly mixing) which satisfies a PNT.

Our main result establishes a PNT along a subsequence for a class of mixing Kochergin flows on $\T^2$:

\begin{theorem}\label{th:main} There exists a $G_\delta$ dense set $\cD$ of irrationals  such that for every $\alpha\in \cD$ the corresponding Kochergin flow $(T^{\alpha,\gamma}_t)$ with one fixed point at $x_0$ with exponent $\gamma$ satisfies: there exists and a subsequence $(N_k)$ of full upper density such that for every $\psi\in C(\T^2)$, for every $x\in \T^2\setminus\{x_0\}$,
$$
\max_{z\in \{+,-\}}\Big|\sum_{p<N_k}\psi(T^{\alpha,\gamma}_{z\cdot p}(x))\log p-\int_{0}^{N_k}\psi(T^{\alpha,\gamma}_{z\cdot t}(x))dt\Big|=o(N_k).
$$
\end{theorem}

Let $\mu$ denote the area measure preserved by $(T^{\alpha,\gamma}_t)$. In Proposition \ref{prop:forback} we also show that for every $x\in \T^2\setminus\{x_0\}$, 
$$
 \min_{z\in\{+,-\}}\Big|\int_{0}^{N_k}\psi(T^{\alpha,\gamma}_{z\cdot t}(x))dt-N_k\int_{\T^2} \psi\;d\mu\Big|=o(N_k).
$$
for every $\psi\in C(\T)$. 

As a corollary we get\footnote{If $a<0$, then $p\in [0,a]$ means $-p\in[a,0]$.}:
\begin{corollary}\label{cor:eqseq} For every $x\in \T^2\setminus \{x_0\}$
the orbit
$$
\{T_{p}^{\alpha,\gamma}(x)\}_{p\in [0,N_k]}
$$
is equidistributed (and in particular dense) with respect to $\mu$ on $\T^2$.
\end{corollary}

Recall that by \cite{Koc1}, Kochergin flows are mixing for every irrational $\alpha$. Therefore, our result seems to be the first for which a PNT along a subsequence holds for mixing systems. We also emphasize, that in most cases where a PNT was established, it was in fact shown for a special class of functions (like characters for irrational rotations) and one then obtains the general case of continuous functions by approximation. In our case there is no natural class of functions (due to mixing) and so the result is established for every continuous function.

 The main number theoretical input is understanding distribution of polynomial (quadratic) phases over primes in short intervals. Recall that Matom{\"a}ki and Shao, \cite{Mat-Shao}, obtained such results for all intervals $[N,N+H]$, where $H\geq N^{2/3+\epsilon}$. In the proof of Theorem \ref{th:main} it is crucial to get below the threshold of $2/3$ (for quadratic polynomials). Such a result was recently obtained in \cite{KLR}, where the authors show cancelations in all intervals of size $H\geq N^{2/3-\eta}$ (with $\eta>0$ small enough). This result allows us to construct (using Proposition \ref{prop:cruc}) the set $\cD$ from Theorem \ref{th:main}. 
 
 We finally mention that in the current form Theorem \ref{th:main} is non-quantitative. One of the main reasons for that is that the result in \cite{KLR} (see Proposition \ref{prop:cruc}) is non-quantitative.

We also study prime orbits for weakly mixing (non-mixing) reparametrizations of linear flows on $\T^2$ (with no fixed points). Our second main result establishes a {\em quantitative} PNT for some weakly mixing systems on the torus:

\begin{theorem}\label{th:maa}For every $A>0$ there exists a class $\mathscr{C}_A$ of $C^{\infty}$, uniquely ergodic and weakly mixing flows on $\T^2$ such that for every $\mathcal{T}=(T_t)\in \mathscr{C}_A$ the following holds: for every $\psi$ in a dense set (in the uniform topology) and every $x\in \T^2$,
$$
\Big|\sum_{p\leq N}\psi(T_px)\log p-  N\int_{\T^2}\psi d\mu\Big|\ll_{A,\psi} N\cdot\log^{-A}N,
$$ 
where $\mu$ is the unique $\mathcal{T}$-invariant measure. Therefore a  PNT holds for  $\mathcal{T}$.
\end{theorem}
  
 Our argument for Theorem \ref{th:maa} relies on the elliptic (periodic) structure of the system along with results on 
 distribution of primes in arithmetic progressions in short intervals to large moduli. One of the main number theoretic input is the Bombieri-Vinogradov theorem.
 
It seems to the best of our knowledge that there is no known example of a mixing dynamical system for which a PNT holds. Therefore the following problem is natural:

\begin{problem}Find an example of a mixing system which satisfies a prime number theorem. 
\end{problem}

\paragraph{Outline of the paper:} In Section \ref{sec:rep} we introduce the basic concepts for the paper, i.e. special flows, reparametrizations and we also restate Theorems \ref{th:main} and \ref{th:maa} in the language of special flows (see Theorems \ref{thm:main2} and \ref{th:maa2}). In Section \ref{sec:dis} we recall important results on distribution of primes in short intervals. We first prove Theorem \ref{th:maa2} in Section \ref{sec:sex}. The proof of Theorem \ref{thm:main2} is done in Section \ref{sec:main2}.

\paragraph{Acknowledgements:} The author would like to thank Maksym Radziwi\l\l  \;\;for his important remarks and sugguestions on Section \ref{sec:2.3}. The research of the author was partially supported by the NSF grant DMS-1956310.

\section{Smooth flows on $\T^2$ and their special representation}\label{sec:rep}
\subsection{Special flows}
For $T\in Aut(X,\mu,d)$ and $g\in L^1(X,\mu)$, $g>0$ we define the {\em special flow} $T^g$ (over $T$ and under the roof function $g$) acting on $X^g=\{(x,s)\;:\; x\in X, 0\leq s<g(x)\}$ by
$$
T_t(x,s):=(x+N(x,s,t)\alpha,s+t-S_{N(x,s,t)}(g)(x)),
$$
where $N(x,s,t)\in \Z$ is unique satisfying $s+t-S_{N(x,s,t)}(g)(x)\in [0,g(T^{N(x,s,t)}x)$ and where
$$
S_n(g)(x)=\sum_{0\leq i<n}g(T^ix),
$$
for $n\geq 0$ and $S_n(g)(x)=S_{-n}(g)(T^{n}x)$ for $n<0$. Note that $T^g$ preserves the measure $\mu^g:=\mu\times {\rm Leb}_{\R}$ restricted to $X^g$. We will also consider the product metric $d^g((x,s),(y,s'))=d(x,y)+|s-s'|$. If it is clear from the context, we will denote $d^g$ by $d$.

\subsection{Reparametrizations of linear flows on $\T^2$}
In this section we will introduce the class of smooth and weakly mixing flows on $\T^2$ for which Theorem \ref{th:maa} holds. The class will be given by reparametrizations (or time-changes) of linear flows on $\T^2$. In fact, \cite{COFOSI}, any fixed point free flow on $\T^2$ is a reparametrization of some linear flow. Below we define the notion of reparametrization. We will do it only for the case of linear flows on $\T^2$, although the definitions are still valid for any abstract measure preserving flow. 

Let $\alpha\in \R\setminus \Q$ and let $L^\alpha_t(x)=x+(t\alpha,t)$ be a linear flow on $\T^2$ in direction $\alpha$. Then $(L^\alpha_t)$ is uniquely ergodic and preserves the Lebesgue measure on $\T^2$ (which we denote $Leb_2$). Let $v\in C^\infty(\T^2)$, $v>0$. We define the {\em cocycle}
\be\label{v1}
v(t,x):=\int_0^tv(L^\alpha_sx)ds.
\ee
Let $u=u(t,x)$ be unique such that 
$$
\int_0^uv(L^\alpha_sx)ds=t
$$
(one can show that such unique $u$ exists). We then define the reparametrization of $(L^\alpha_t)$ given by $v$:
$$
T^{\alpha,v}_t(x):=L^{\alpha}_{u(t,x)}(x).
$$
Note that $(T^{\alpha,v}_t)$ has the same orbits as $(L^\alpha_t)$. Moreover, $(T^{\alpha,v}_t)$ is uniquely ergodic with the measure $\mu$ given by $d\mu=\Big(\frac{v}{\int_{\T^2}v dLeb_2}\Big)dLeb_2$. We additionally assume that $\int_{\T^2}v dLeb_2=1$. We say that $v$ is a {\em quasi-coboundary} if 
$$
v(t,x)-t=h(x)-h(L^{\alpha}_tx),
$$
for some measurable $h:\T^2\to \R$. We say that $v$ is a smooth quasi coboundary, if $h$ is smooth.  If $w\in C^\infty(\T^2)$ is another time change, then $v$ and $w$ are {\em cohomologous}, if 
$$
 v(t,x)-w(t,x)=h(x)-h(L^\alpha_tx).
$$
If $v$ and $w$ are cohomologous, then $(T^{\alpha,v}_t)$ and  $(T^{\alpha,w}_t)$ are {\em isomorphic} with the isomorphism given by $S(x)=T^{\alpha,w}_{h(x)}(x)$.

\subsection{Theorem \ref{th:maa}}\label{sec:2.3}
We will now specify the class $\mathscr{C}_A$ from Theorem \ref{th:maa}. First we need some properties on distribution of primes in residue classes to large moduli.
\subsubsection{Primes in residue classes to large moduli}
The set $\mathcal{P}$ always denotes the set of prime numbers. The main result of this section is:
\begin{lemma}\label{lem:nt6}Fix $A>10$. For any $q,r\in \mathcal{P}$ with $r\geq e^{A^{10}}$ and $q\in[\frac{1}{2}e^{r^{A^{-3}}},e^{r^{A^{-3}}}]$, there  exists a set $S(q,r)\subset \mathcal{P}\cap [\frac{1}{2}e^{q^{A^{-3}}},e^{q^{A^{-3}}}]$ such that:
\begin{enumerate}
\item[P0.] $|S(q,r)|\geq \frac{e^{q^{A^{-3}}}}{6rq^{A^{-3}}}$
\item[P1.] $\ell\equiv r \mod q$,  for $\ell\in S(q,r)$.
\item[P2.] For every $\ell\in S(q,r)$ and all $x$ satisfying $\ell\leq x^{1/2+1/100}$,
$$
\max_{a<\ell}\Big|\sum_{\substack{p\leq x\\p\equiv a \mod \ell}}\log p-\frac{x}{\ell}\Big|\leq \frac{2Cx}{\ell \log^Ax}.
$$

\end{enumerate}
\end{lemma}
\begin{proof}First, as a  consequence of the Bombieri-Vinogradov theorem, \cite{Bom},\cite{Vi2}, we have: for every $A>0$ there exists $C>10$ such that \begin{equation}\label{eq:BV}
\sum_{1\leq q\leq x^{1/2-1000}}\max_{(a,q)=1}\sup_{y<x}\Big|\sum_{\substack{p\leq y\\p\equiv a \mod q}}\log p-\frac{y}{\varphi(q)}\Big|\leq \frac{C x}{\log^{10A}x}.
\end{equation}
Let $$E(x,q):=\max_{(a,q)=1}\sup_{y<x}\Big|\sum_{\substack{p\leq y\\p\equiv a \mod q}}\log p-\frac{y}{\varphi(q)}\Big|$$ and let
$$
Z_N(x):=\Big\{N/2\leq q\leq N, q\text{ is prime }\;:\; E(x,q)\leq \frac{Cx}{N\log^{2A}x}\Big\}
$$

 It follows from \eqref{eq:BV} and the prime number theorem that 
\begin{equation}\label{eq:upbou}
|Z_N(x)|\geq [\pi(N)-\pi(N/2)]-\frac{N}{\log^{8A}x}.
\end{equation}
Let $x_1=N^{1/2+1/100}$ and $x_n=2x_{n-1}$. Notice that if $q\in Z_N(x_n))$, then
$$
\max_{(a,q)=1}\sup_{x_n/2<y<x_n}\Big|\sum_{\substack{p\leq y\\p\equiv a \mod q}}\log p-\frac{y}{\varphi(q)}\Big|\leq \frac{Cx}{N\log^{2A}x}.
$$
In particular, if $q\in \in\bigcap_{n}Z_N(x_n)$, then by the definition of $x_n$, for every $x\geq N^{1/2+100}$,
\begin{equation}\label{eq:zz6}
\max_{(a,q)=1}\Big|\sum_{\substack{p\leq x\\p\equiv a \mod q}}\log p-\frac{x}{\varphi(q)}\Big|\leq \frac{2Cx}{N\log^{2A}x}.
\end{equation}
Notice that
$$ 
\bigcup_{n\in \N}\Big|[N/2,N]\cap \mathcal{P}\cap Z_N(x_n)^c\Big|\leq N\cdot \sum_{n\geq 1}\frac{1}{\log^{8A}x_n}\leq N\cdot \sum_{n\geq 1}\frac{1}{\log^{8A}x_1+n^{8A}}\leq \frac{20N}{\log^{4A}N}.
$$
Therefore and by the prime number theorem, 
\begin{equation}\label{eq:cl}
\Big|\bigcap_{n\in \N}Z_N(x_n)\Big|\geq \frac{N}{2\log N}-\frac{20N}{\log^{4}N}.
\end{equation}

Let $q,r$ be given as in the statement and let $N:=[e^{q^{A^{-3}}}]$. Define $K(q,r):=\{p\in [N/2,N]\;:\; p\equiv r\mod q\}$. By the Siegel-Walfisz theorem, 
$$
|K(q,r)|\geq \frac{N}{3r\log N}.
$$
Define $S(q,r):=[\bigcap_{n}Z_N(x_n)]\cap K(q,r)$. Then by \eqref{eq:cl} and the bound on $r$,
 $|S(q,r)|\geq \frac{N}{3r\log N}-\frac{20N}{\log^{4}N}\geq \frac{N}{6r \log N}$ and so $P0$ holds. Notice that $P1$ immediately follows from $S(q,r)\subset K(q,r)$, and  $P2$ immediately follows from $S(q,r)\subset \bigcap_{n}Z_N(x_n)$ and \eqref{eq:zz6}. 
 \end{proof}

\subsubsection{Statement of Theorem 1.3.}

For $\alpha\in \R\setminus \Q$, let $(q_n)$ denote the sequence of denominators of $\alpha$. With Lemma \ref{lem:nt6}, we have the following lemma:
\begin{lemma}\label{lem:nt2} For every $A>0$ there exists an uncountable set $C_A$ such that for every $\alpha\in C_A$ the sequence $(q_n)$ satisfies: there exists $n_0$ such that for every $n\geq n_0$,
\begin{enumerate}
\item[H1.]  $q_n$ is a prime number;
\item[H2.] $q_{n+1}\in\Big[\frac{1}{2}e^{q_n^{A^{-3}}},e^{q_n^{A^{-3}}}\Big]$;
\item[H3.] For every $n\geq n_0$ and every $x\in \N$ satisfying $x^{1/2+1/100}\geq q_n$, we have
$$
\max_{a<q_n}\Big|\sum_{\substack{p\leq x\\p\equiv a \mod q_n}}\log p-\frac{N}{q_n}\Big|\leq \frac{2Cx}{q_n \log^Ax},
$$
for some constant $C>0$.
\end{enumerate}
\end{lemma}
\begin{proof} Recall that $q_{n+1}=a_nq_n+q_{n-1}$. Inductively, having defined $q_n$ and $q_{n-1}$, the only restriction that we have in defining $q_{n+1}$ is $q_{n+1}\equiv q_{n-1} \mod q_n$ (see $P1.$). We  pick $a_n$ so that the corresponding $q_{n+1}$ satisfies $q_{n+1}\in S(q_n,q_{n-1})$. Then $H1.$ and $H2.$ follow from $P0.$ and the definition of $S(q_n,q_{n-1})$, and $H3.$ follows immediately from $P2.$. Moreover, by the bound on the cardinality of $S(q_n,q_{n-1})$, the set $C_A$ is uncountable.
\end{proof}

Let $e(x)=e^{2\pi i x}$. For $w\in C^\infty(\T^2)$, let $w=\sum_{(k,\ell)\in \Z^2}a_{k,\ell}e(kx+\ell y)$. For $\alpha\in C_A$, let 
\begin{multline*}
\mathscr{C}_A^\alpha:=\\
\Big\{v\in C^\infty: v(x,y)=1+Re\Big(\sum_{n}\sum_ma_{q_n,m}e(q_nx+my)\Big),\text{ where } |a_{q_n,0}|\in[q_{n+1}^{-2/3},q_{n+1}^{-1/2}]\Big\}
\end{multline*}
and 
$$
\mathscr{C}_A=\bigcup_{\alpha\in C_A}\mathscr{C}_A^\alpha.
$$
Let $v\in \mathscr{C}_A$ and let $(T^v_t)$ be the reparametrization of $(L^\alpha_t)$ given by $v$. Let $Cob_v\subset C^1(\T^2)$ denote 
the set of smooth quasi-coboundaries for the automorphism $T^v_1$.

We can now restate Theorem \ref{th:maa} as follows:
\begin{theorem}\label{th:maa2}Let $v\in \mathscr{C}_A$. Then $(T^v_t,\T^2,\mu)$ is weakly mixing and for any $\psi\in Cob_v$ and every $(x,y)\in \T^2$, 
$$
\Big|\sum_{p\leq N}\psi(T^v_p(x,y))\log p- N\int_{\T^2}\psi\, d\mu\Big|\ll_{A,\psi}N\log^{-A}N.
$$
\end{theorem}
We will prove the above theorem in Section \ref{sec:sex}.

\begin{remark} It follows from unique ergodicity that the space $Cob_v$ is dense in $C^1(\T^2)$ in the uniform topology. Indeed, it is enough to notice that for every mean zero $g\in C^1(T^2)$
$$
 \frac1n\sum_{n\leq N}(S_n(g\circ T^v_1)-S_n(g))=-g+\frac1N\sum_{n\leq N}g\circ (T^v_1)^n.
$$
Now the LHS is a coboundary and the RHS is uniformly close to $-g$ by unique ergodicity.
\end{remark}

We have the following lemma (see \eqref{v1}). 
\begin{lemma}\label{lem:v1} Let $v\in \mathscr{C}_A$. Let $f(x):=v(1,(x,0))$. Then $f\in C^\infty(\T)$ and $f(x)=1+Re(\sum_{n}b_{q_n}e(q_nx))$, where 
$$
|b_{q_n}|=|a_{q_n,0}|\in [q_{n+1}^{-2/3},q_{n+1}^{-1/2}].
$$
\end{lemma}
\begin{proof} Note that $f(x)=\int_0^1v(L_s^\alpha(x,0))ds=\int_0^1v(x,s)ds$. We have 
$$ 
\int_0^1v(x,s)ds=1+\sum_{n}\sum_ma_{q_n,m}\int_{0}^1Re\Big(e(q_nx+ms)\Big)ds.
$$
It remains to notice that for $m\neq 0$, $\int_{0}^1Re\Big(e(q_nx+ms)\Big)ds=0$ and $\int_{0}^1Re\Big(e(q_nx)\Big)ds=Re\Big(e(q_nx)\Big)$.
This finishes the proof.
\end{proof}

\subsection{Special representation of Kochergin flows}
We will study Kochergin flows $T^{\alpha,\gamma}$ via their {\em special representation}. As the orbits are qualitatively orbits of the  linear flow in direction $(\alpha,1)$, it is natural to take $\T\times\{0\}$ as a section for the flow. Assume WLOG that the singularity lies on the orbit of $(0,0)$. Notice that the {\em first return map}  is just an irrational rotation by $\alpha$ (except the segment joining  $(0,0)$ with $x_0$). Moreover, the first return map $f$ is smooth except $0\in \T$ at which it blows up to infinity -- the closer the orbit passes to the singularity, the longer it takes to come back.  In particular, as shown by Kochergin \cite{Koc1}, the roof function $f$ satisfies: for $\gamma\in (-1,0)$ and $i=0,1,2$
\begin{equation}\label{eq:asu}
\lim_{x\to 0^{+}}\frac{\partial^i f(x)}{x^{-i+\gamma}}=A_i\;\;\text{ and } 
\;\;\lim_{x\to 0^{-}}\frac{\partial^i f(x)}{(1-x)^{-i+\gamma}}=B_i, 
\end{equation}
where $\partial^0f =f$ and $A_0,A_2,B_0,B_1,B_2>0$ and $A_1<0$. We say that $f\in C^{2}(\T\setminus\{0\})$ satisfying \eqref{eq:asu} has {\em power singularity with exponent $\gamma$}. We will denote the special flow corresponding to the Kochergin flow by $T^{\alpha,\gamma}$ or $T^\alpha$ since the roof function $f$ and hence also $\gamma$ is fixed. 
\begin{remark}\label{rem:ret} From the above representation it follows that every point on $\T^2$, except points on the segment joining $x_0$ with $(\alpha,1)$, has a unique representative in the special representation. Moreover, fix a point $q\neq x_0$ on the segment from $x_0$ to $(\alpha,1)$. Then every point $T^{\alpha}_t q$ with $t$ larger than the return time of $q$ to the transversal has a unique representative in the special representation.
\end{remark}

Let $\psi\in C(\T^2)$ and let $\psi_\infty:=\psi(x_0)$ be the value of $\psi$ at the fixed point. Notice that $\psi$ corresponds to the following function $\bar{\psi}$ in the special representation: 
\begin{itemize}
\item[i.] $\bar{\psi}\in C(\T^f)$;
\item[ii.] $\bar{\psi}(y,f(y)):=\lim_{s\to f(y)^-}\bar{\psi}(y,s)$ is equal to $\bar{\psi}(y+\alpha,0)$.
\item[iii.] for every $\epsilon>0$ there exists $\delta>0$ such that for every $(y,r)\in \{(x,s)\in \T^f\:\; s\geq \frac{1}{\delta}\}$, $|\bar{\psi}(y,r)-\psi_\infty|<\epsilon$.
\end{itemize}

The second property follows from the fact that the point $(y+\alpha,0)$ correponds to $(y+\alpha,0)$ in the special representation and $(y+\alpha,1)$ (indetified with $(y+\alpha,0)$ on $\T^2$) corresponds to $(y,f(y))$ in the special representation. The third property follows from the fact that the set $(y,r)\in \{(x,s)\in \T^f\:\;\|x\|<\delta, s\geq \frac{1}{\delta}\}$ corresponds to a neighborhood of the fixed point $x_0\in \T^2$ and so values of $\psi$ converge to $\psi_\infty$ when moving towards the fixed point. 

Using the special representation of the Kochergin flow (see also Remark \ref{rem:ret}), Theorem \ref{th:main} is a consequence of the following theorem: 
\begin{theorem}\label{thm:main2} There exists a $G_\delta$ dense set $\cD$ such that for every $\alpha\in \cD$ there exists $c>0$ and a subsequence of denominators $(q_{n_k})$ of $\alpha$ such that for every $\bar{\psi}$ satisfying i.-iii., we have: for every $(x,s)\in \T^f$, every $m\geq 0$ and every $N_k\in [\frac{q_{n_k+1}}{\log k},cq_{n_k+1}]$, 
$$
\max_{z\in\{+,-\}}\Big|\sum_{p<N_k}\bar{\psi}(T^{\alpha}_{z\cdot (p-m)}(x,s))\log p-\int_{0}^{N_k}\bar{\psi}(T^{\alpha}_{z\cdot t}(x,s))dt\Big|=o(N_k).
$$
\end{theorem}
We will now show how to deduce Theorem \ref{th:main} from the above theorem.
\begin{proof}[Theorem \ref{thm:main2} implies Theorem \ref{th:main}]
Notice first that the set $\bigcup_{k}[\frac{q_{n_k+1}}{\log k},\frac{q_{n_k+1}}{10}]\subset \N$ has full upper density.
For every point $q\in \T^2$ for which there exists a unique representative $(x_q,s_q)\in \T^f$, Theorem \ref{th:main} follows immediately from Theorem \ref{thm:main2} with $m=0$. So it remains to consider points $q\in \T^2\setminus\{x_0\}$ which do not have a represenative. These are precisely points lying on the seqment joining $x_0$ and $(\alpha,1)$. Fix any such $q$. Let $m=m_q>0$ be the first return time to the section. Then 
$$
\sum_{p\leq N_k}\psi(T_p^\alpha(q))=\sum_{p\leq N_k}\psi(T_{p-m}^\alpha(T^\alpha_mq)),
$$
and 
$$
\int_{0}^{N_k}\psi(T^{\alpha}_{t}(q)dt=\int_{-m}^{-m-N_k}\psi(T^{\alpha}_{t}(T^\alpha_m q))dt=\int_{0}^{N_k}\psi(T^{\alpha}_{t}(T^\alpha_m q))dt+{\rm O}(m).
$$
Then  $T^\alpha_mq\in \T^2$ has a unique representative $(x_{m,q},s_{m,q})\in \T^f$ in the special representation so we can use Theorem \ref{thm:main2} form $(x_{m,q},s_{m,q})$ to deduce Theorem \ref{th:main} for $q$.
\end{proof}

We will also prove the following proposition, which together with Theorem \ref{thm:main2} implies that prime orbits are equidistributed along the subsequence $(N_k)$.
\begin{proposition}\label{prop:forback} Let $\alpha\in \cD$. There exists $c>0$ such that for every $(x,s)\in \T^f$, $T\in [\frac{q_{n_k+1}}{\log k},cq_{n_k+1}]$ (where $(q_{n_k})$ is the subsequence from Theorem \ref{thm:main2}),
\be\label{eq:forback}
 \min_{z\in\{+,-\}}\Big|\int_{0}^{T}\bar{\psi}(T^{\alpha}_{z\cdot t}(x,s))dt-T\int_{\T^f} \bar{\psi}\;d\,{\rm Leb}^f\Big|=o(T).
\ee
\end{proposition}

\section{Distribution of primes in short intervals}\label{sec:dis}
We first recall the following result:
\begin{lemma}[\cite{OPP}]\label{lem:shit}Let $\varepsilon>0$ and $A>0$ be given. Then for $(a,r)=1$, $r\leq (\log N)^A$ and $H>N^{7/12+\epsilon}$,
$$
\sum_{\substack{p\in [N,N+H]\\p\equiv a\; \mod\; r}}\log p=\frac{H}{\varphi(r)}+{\rm O}_{A,\epsilon}\Big(\frac{H}{\varphi(r)(\log N)^A}\Big)
$$
\end{lemma}
This immediately implies the following:
\begin{remark}\label{rem:prim} Let $I\subset N$ be an interval. Then 
$$
|\sum_{p\in I}\log p|\ll \max(|I|,N^{2/3}).
$$
Indeed, if $|I|\geq N^{7/12+1/100}$, we use the above lemma. Otherwise we trivially bound the above sum by $N^{7/12+1/100}\log N\leq N^{2/3}$.
\end{remark}
In fact we have a stronger estimate:
\begin{lemma}[Brun-Titchmarsh theorem]\label{lem:fff}Let $I\subset [0,N]$ be an interval with $|I|\geq N^{1/10}$. Then 
$$
\sum_{p\in I}\log p\ll |I|.
$$
\end{lemma}

We also have the following lemma (Huxley to small moduli):
\begin{lemma}\label{lem:nt3}Fix $\epsilon>0$ and $A>4$. For every $N\in \N$, $H\geq N^{1/6+\epsilon}$, we have 
$$
\sum_{y\leq N}\sup_{(a,v)=1}\Big|\sum_{\substack{p\in [y,y+H]\\  p\equiv a \pmod v}}\log p-\frac{H}{\varphi(v)}\Big|\ll_{A,\epsilon}\frac{HN}{\log^AN},
$$
uniformly over $v\leq \log^A N$.
\end{lemma}
\begin{proof}
This follows by taking $Q = (\log x)^{A}$ in \cite[Theorem 1.2]{Koukoupulos}  (see also Lemma 8.14 in \cite{KLR}).
\end{proof}

\begin{corollary}\label{cor:ua} Under the assumptions of the above lemma, there exists $z\leq H$ such that 
\be\label{eq:zz}
\sum_{j=0}^{[\frac{N}{H}]-1}\sup_{(a,v)=1}\Big|\sum_{\substack{p\in [z+jH,z+(j+1)H]\\  p\equiv a \pmod v}}\log p-\frac{H}{\varphi(v)}\Big|\ll_{A,\epsilon} \frac{N}{\log^{A}N}.
\ee
\end{corollary}
\begin{proof}The proof follows from pigeonhole principle by considering residue classes $\mod\, H$.
\end{proof}

For the proof of Theorem \ref{thm:main2} we will crucially need equidistribution of quadratic polynomials on primes in short intervals.
 It is crucial for the paper that  the length of the short interval $[N,N+H]$ on which we control averages over primes is of order $H\sim N^{\theta}$, with $\theta<2/3$. Recall that \cite{Mat-Shao} obtained such results for $\theta>2/3$. In a recent paper, \cite{KLR}, the authors were able to weaken results of \cite{Mat-Shao} below $2/3$ (see Theorem 10.1. and Propositions 10.2 and 10.3). In the result below we use the results obtained in \cite{KLR}. We have:
\begin{proposition}[Proposition 10.2., \cite{KLR}]\label{prop:n1}Let $\eta\in (0,10^{-7})$ be given. There exists $N_0=N_0(\eta)$ such that  for $N>H>N^{2/3-\eta}>N_0$ and $g(n)=\gamma_1(n-N)+\gamma_2(n-N)^2$ the following holds: if for all $0<r\leq (\log N)^{B'}$ with $B'$ sufficiently large in terms of $1/\eta$, there exists an $i\in\{1,2\}$ such that 
$$
\|r\gamma_i\|\geq \frac{(\log N)^{B'}}{H^i}, 
$$ 
then 
$$
\Big|\sum_{p\in [N,N+H]}e(g(p))\log p\Big|\ll \eta^{2/3} H.
$$
\end{proposition}

As mentioned above, the weakening below $2/3$ is crucial for the proof of Theorem \ref{thm:main2}.

\section{Proof of Theorem \ref{th:maa2}}\label{sec:sex}

\subsection{Weak mixing of the reparametrized flow}
Let $v\in \mathscr{C}_A$, $f(x):=v(1,(x,0))$. It follows by a result of Katok, \cite{Kat44}, that sufficient conditions (in the language of special flows) for weak mixing of the special flow $(T^v_t,\T^2,\mu)$ are given by:
$$ 
\frac{\Big|\alpha-\frac{p_n}{q_n}\Big|q_n}{|\hat{f}(q_n)|}\to 0,
$$
and 
$$
\frac{|\hat{f}(q_n)|}{\sum_{k=1}^{+\infty}|\hat{f}(kq_n)|}>c>0.
$$
Notice that the first condition is satisfied for $v \in  \mathscr{C}_A$: the only non-zero frequencies are coming from denominators of $\alpha$ and by Lemma \ref{lem:v1},
$$
\frac{\Big|\alpha-\frac{p_n}{q_n}\Big|q_n}{\hat{f}(q_n)}\leq Cq_{n+1}^{2/3}\cdot q_{n+1}^{-1}\cdot q_n\to 0,
$$
since $\alpha\in C_A$.  Moreover, the second condition is also satisfied, as by Lemma \ref{lem:v1} and the fact that $\alpha\in C_A$,
$$
\sum_{k\geq 2}|\hat{f}(kq_n)|\leq \sum_{k\geq q_{n+1}}|\hat{f}(k)|\ll
q_{n+2}^{-1/2}<\hat{f}(q_n).
$$
Therefore $(T^v_t)$ is weakly mixing as claimed.
\subsection{Quantitative rigidity of smooth cocycles over rotations}
Let $\alpha\in C_A$ with the sequence of denominators $(q_n)$ and $v\in \mathscr{C}_A$. Denote $f(x):= v(1,(x,0))$. Recall that $\int_\T f dLeb_1=1$.

\begin{lemma}\label{lem:rig}For every $x\in \T$ and every $0\leq k\leq \frac{q_{n+1}^{1/2-1/100}}{q_n}$,
$$
\Big|S_{kq_n}(f)(x)-kq_n\Big|\ll_A q_{n}^{-3}.
$$
\end{lemma}
\begin{proof} The proof goes by standard Fourier analysis arguments.  Note that if $f_n(x)=Re(\sum_{k>n}b_{q_k}e(q_kx))$, then by Lemma \ref{lem:v1}, 
$$
|S_{kq_n}(f_n)(x)|\leq kq_n(\sup |f_n|)\ll q_{n+1}^{1/2}q_{n+2}^{1/2}\ll q_n^{-3}.
$$
Moreover for $1\leq m<n$,
$$
|S_{q_n}(b_{q_m}e(q_m x))|\ll |S_{q_n}(e(q_mx))|\leq \Big|\frac{e(q_nq_m\alpha)-1}{e(q_m\alpha)-1}\Big|\ll \frac{q_{m+1}q_m}{q_{n+1}}.
$$
Therefore, by cocycle identity, and the bound on $k$,
$$
|S_{kq_n}(b_{q_m}e(q_m x))|\ll \frac{q_{m+1}^2}{q_{n+1}^{-1/2}}.
$$
So 
$$
|S_{kq_n}(\sum_{m<n}b_{q_m}e(q_m x))|\ll nq_{n}^2 q_{n+1}^{-1/2}\ll q_n^{-3},
$$
since $\alpha\in C_A$
Finally, for $m=n$ by cocycle identity and the bound on $a_{q_n}$,
$$
|S_{kq_n}(b_{q_n}e(q_n x))|\leq q_{n+1}^{1/2-1/100}q_{n+1}^{-1/2}\ll_A q_n^{-3},
$$
since $\alpha \in C_A$. 

This finishes the proof.

\end{proof}
We have the following corollary:
\begin{corollary}\label{cor:rig} For every $0\leq k\leq \frac{q_{n+1}^{1/2-1/100}}{q_n}$, every $a<q_n$ and every $(x,y)\in \T^2$, 
$$
d\Big(T^v_{kq_n+a}(x,y),T^v_a(x,y)\Big)\ll_A q_n^{-2}.
$$
\end{corollary}
\begin{proof}Notice that the statement follows by showing that for every $(x,y)\in \T^2$,
$$
d\Big(T^v_{kq_n}(x,y),(x,y)\Big)\ll_A q_n^{-2}.
$$
By the smoothness of the flow, we have 
$$
d\Big(T^v_{kq_n}(x,y), (x,y)\Big)\ll d\Big(T^v_{kq_n}(x,0),(x,0)\Big).
$$
Since $f(x)=\int_0^1v(L_s^\alpha(x,0))ds$, it follows that $S_{kq_n}(f)(x)=\int_0^{kq_n}v(L_s^\alpha(x,0))ds$. Therefore and by the definition of reparametrization, 
$$
T^v_{S_{kq_n}(f)(x)}(x,0)=L_{kq_n}^\alpha(x,0)=(x+kq_n\alpha,0).
$$
By the bound on $k$, $\|kq_n\alpha\|\ll q_{n+1}^{-1/2}\leq q_n^{-10}$. This together with Lemma \ref{lem:rig} and smoothness of the flow $(T_t^v)$ finishes the proof.
\end{proof}

\subsection{Proof of Theorem \ref{th:maa2}}
Since $v\in \mathscr{C}_A$ and $\alpha\in C_A$ are now fixed, we will denote the corresponding flow simply by $(T_t)$. We will WLOG assume that $\psi\in Cob_v$ satisfies $\int_{\T^2}\psi\, d\mu=0$. This in particular means that 
\be\label{eq:cob}
\frac{1}{M}\sum_{n\leq M}\psi(T_n(x,y))={\rm O}_\psi(1),
\ee
for every $M\in \N$.

Let $N\in \N$ and let $n\in \N$ be unique such that 
$$
q_n^{3-1/10}\leq N<q_{n+1}^{3-1/10}.
$$
We will consider two cases:\\

\textbf{Case 1.} $N\in [q_n^{3-1/10}, q_{n+1}^{1/2-1/100}]$.
In this case, we write 
$$
\sum_{p\leq N}\psi(T_p(x,y))\log p=\sum_{a<q_n}\sum_{\substack{p\leq N\\p\equiv a\mod q_n}}\psi(T_p(x,y))\log p.
$$
By Corollary \ref{cor:rig}, it follows that for such $p\leq N$, $p\equiv a\mod q_n$, 
$$
d(T_p(x,y),T_a(x,y))\ll_{A}q_n^{-2}.
$$
Therefore, since $\psi \in C^1(\T^2)$ and by the prime number theorem, 
$$
\sum_{a<q_n}\sum_{\substack{p\leq N\\p\equiv a\mod q_n}}\psi(T_p(x,y))\log p=\sum_{a<q_n}\psi(T_a(x,y))\cdot\Big(\sum_{\substack{p\leq N\\p\equiv a\mod q_n}}\log p \Big)+$$$${\rm O}_{A}(N\cdot q_n^{-2}).
$$
Since $\alpha\in C_A$,  and $N\geq q_n^{3-1/10}$, by $H3$ in  Lemma \ref{lem:nt2} (with $x=N$), 
$$
\sum_{\substack{p\leq N\\p\equiv a\mod q_n}}\log p =\frac{N}{q_n}+{\rm O}_A(N\cdot q_n^{-1} \cdot \log^{-A}N).
$$
Therefore, 
$$
\sum_{p\leq N}\psi(T_p(x,y))\log p=Nq_n^{-1}\sum_{a<q_n}\psi(T_a(x,y))+{\rm O}_{A}(N\cdot q_n^{-2})+ {\rm O}_A(N\cdot \log^{-A}N).
$$
Note that by \eqref{eq:cob} it follows that 
$$
q_n^{-1}\sum_{a<q_n}\psi(T_a(x,y))\ll_\psi q_n^{-1}.
$$
Therefore and using $q_n>\log^A q_{n+1}>\log^A N$, we get 
$$
\sum_{p\leq N}\psi(T_p(x,y))\log p={\rm O}_A(N\cdot \log^{-A}N).
$$
This finishes the proof in this case.\\

\textbf{Case 2.} $N\in [q_{n+1}^{1/2-1/100},q_{n+1}^{3-1/10}]$. Let $H:=N^{1/6+\epsilon_0}$, so that $(3-1/10)\cdot (1/6+\epsilon_0)<1/2-1/100$. Note that since $\alpha\in C_A$, 
$e^{q_n^{A^{-3}}}\leq 2q_{n+1}$ and so 
$$
q_n\ll \log^{A^3}q_{n+1}\ll \log^{A^3}N. 
$$
We will use Corollary \ref{cor:ua} for $\epsilon_0$ and $A^4$ and $v=q_n$.
We divide the interval $[0,N]$ into intervals $I_j=[z+jH, z+(j+1)H)$, $j\leq \Big[\frac{N}{H}\Big]-1$, where $z<H$ comes from Corollary \ref{cor:ua}. 
 We will WLOG assume that $z=0$ (the argument in the general case follows analogous steps). We will say that $I_j$ is {\em good} if 
$$
\sup_{(a,q_n)=1}\Big|\sum_{\substack{p\leq I_j\\ p\equiv a\mod q_n}}\log p-\frac{H}{q_n}\Big|\ll_{A^4,\epsilon_0} \frac{H}{\log^{A^2}N},
$$
otherwise $I_j$ is not good. Notice that by Corollary \ref{cor:ua} the cardinality of non-good intervals is bounded above by $\frac{N}{H\log^{A^2}N}$.  Therefore and by Lemma \ref{lem:fff}, 
$$
\sum_{j \text{ not good}}\sum_{p\in I_j}\psi(T^px)\log p\ll\frac{N}{\log^{A^2}N},
$$
and therefore we can focus only on good intervals. For good $I_j=[b_j,c_j]$, let $(x_j,y_j)=T_{b_j}(x,y)$. Then
 
$$
\sum_{p\in I_j}\psi(T_p(x,y))\log p=\sum_{p\in I_j}\psi(T_{p-b_j}(x_j,y_j))\log p.
$$
Note that $p-b_j\leq H\leq N^{1/6+\epsilon_0}\leq q_{n+1}^{1/2-1/100}$. Therefore, by Corollary \ref{cor:rig} if $p-b_j\equiv a\mod q_n$, then 
$$
d\Big(T_{p-b_j}(x_j,y_j),T_a(x_j,y_j)\Big)\ll_A q_n^{-2}.
$$
Therefore, using that $\psi\in C^1(\T^2)$ and since $I_j$ is good,
$$
\sum_{p\in I_j}\psi(T_{p-b_j}(x_j,y_j))\log p=\sum_{a<q_n}\sum_{\substack{p\in I_j\\p\equiv a \mod q_n}}\psi(T_{p-b_j}(x_j,y_j))\log p=$$$$\frac{H}{q_n}\sum_{a<q_n}\psi(T_a(x_j,y_j))+{\rm O}(Hq_n\log^{-A^2}N)+ {\rm O}(Hq_n^{-1}).
$$
By \eqref{eq:cob}, it follows that $\frac{1}{q_n}\sum_{a<q_n}\psi(T_a(x_j,y_j))\ll_\psi q_n^{-1}$. Therefore, and since $q_n\gg \log^{A^3}q_{n+1}\gg \log^{A^3}N$,
$$
\sum_{p\in I_j}\psi(T_{p-b_j}(x_j,y_j))\log p=\frac{H}{\log^AN}.
$$
Summing over good $I_j$ finishes the proof.

\section{Proof of Theorem \ref{thm:main2} and Proposition \ref{prop:forback}}\label{sec:main2}
\subsection{Distribution of polynomial phases over primes}
In this section we study polynomial phases over primes in short intervals. We have the following proposition:
\begin{proposition}\label{prop:cruc}For any $q\in \N$ there exists $H_0(q)\in \N$, $B=B(q)\in \N$ and $\theta(q)>0$ such that for any $H>H_0(q)$, any $N\geq H$ satisfying $N^{2/3-\theta(q)}\leq H$ there exists a collection of disjoint intervals $\{I_i\}_{i=1}^v$ of equal length covering $\T$ with $\frac{1}{2q^2}<|I_i|<\frac{2}{q^2}$ such that for any interval $J\subset \T$
 with $\frac{1}{2q^2}<|J|<\frac{2}{q^2}$ and any $\gamma_1,\gamma_2\in (0,1)$ with $\gamma_2$ satisfying 
\begin{equation}\label{eq:q2}
\|r\gamma_2\|\geq \frac{(\log N)^B}{H^2}\;\text{ for every }\;0<r\leq (\log N)^B
\end{equation}
the following holds:
$$
\frac{1}{H}\sum_{p\in [N,N+H]}\chi_{I\times J}(\gamma_1(p-N),\gamma_2(p-N)^2)\log p= $$$$
\lambda(J)\Big[\frac{1}{H}\sum_{p\in [N,N+H]}\chi_I(\gamma_1(p-N))\log p\Big] +{\rm O}(q^{-6}).
$$
\end{proposition}
\begin{proof}Let $d$ denote the metric on $\T$. For $A\subset \T$, let $V_\delta(A):=\{x\in \T\;:\; d(A,x)<\delta\}$ be the $\delta$ neighborhood of $A$. Let $I,J\subset \T$ be intervals with $\frac{1}{2q^2}<|I|,|J|<\frac{2}{q^2}$.
Let $f^-_I,f^+_I,f^-_J,f^+_J\in C(\T)$ be positive functions bounded above by $1$ such that for $K\in \{I,J\}$,  $f^-_K\leq \chi_{K}\leq f^+_K$ and moreover,
$$
f^{-}_K(x)=\chi_K(x), \text{ for } x\in \T\setminus V_{q^{-9}}(K^c),
$$
and
$f^-_K(x)=0$ for $x\notin K$. Similarly, 
$$
f^{+}_K(x)=\chi_K(x), \text{ for } x\in K,
$$
and
$f^+_K(x)=0$ for $x\notin V_{q^{-9}}(K)$. For $i\in \{-,+\}$, let $f^i(x,y):=f^i_I(x)\cdot f^i_J(y)$. 
Then 
$$
\sum_{p\in [N,N+H]}\chi_{I\times J}(\gamma_1(p-N),\gamma_2(p-N)^2)\log p\geq \sum_{p\in [N,N+H]}f^-(\gamma_1(p-N),\gamma_2(p-N)^2)\log p
$$
and 
$$
\sum_{p\in [N,N+H]}\chi_{I\times J}(\gamma_1(p-N),\gamma_2(p-N)^2)\log p\leq \sum_{p\in [N,N+H]}f^+(\gamma_1(p-N),\gamma_2(p-N)^2)\log p.
$$

Let $i\in \{-,+\}$ and $K\in \{I,J\}$. Since  $f^i_K\in C(\T)$, there exists $m(q)>q$ such that for every $x\in \T$,
\begin{equation}\label{eq:n1}
\Big|f_K^i(x)-\sum_{|a_{i,K}|<m(q)}c_{a_{i,K}}e_{a_{i,K}}(x)\Big|<q^{-9}.
\end{equation}
Therefore, 
$$
\sum_{p\in [N,N+H]}f^{i}(\gamma_1(p-N),\gamma_2(p-N)^2)\log p=
$$$$
\sum_{p\in [N,N+H]}\Big[\Big(\sum_{|a_{i,I}|<m(q)}c_{a_{i,I}}e_{a_{i,I}}(\gamma_1(p-N))\Big)\cdot \Big(\sum_{|a_{i,J}|<m(q)}c_{a_{i,J}}e_{a_{i,J}}(\gamma_2(p-N)^2)\Big)\log p\Big]+$$$$
{\rm O}\Big(q^{-9}\sum_{p\in [N,N+H]}\log p\Big)
$$
We split the first term according to the value of $a_{i,J}$:
\begin{equation}\label{eq:split}
\sum_{|a_{i,J}|<m(q)}\sum_{|a_{i,I}|<m(q)}\sum_{p\in [N,N+H]}c_{a_{i,I}}c_{a_{i,J}}e\Big(a_{i,J}\gamma_2(p-N)^2+a_{i,I}\gamma_1(p-N)\Big)\log p.
\end{equation}

We use Proposition \ref{prop:n1} with $\eta:=m(q)^{-24}>0$. It implies that if $H\geq H'(q)$, $N\geq H\geq N^{2/3-\eta}$ and $a_{i,J}\neq 0$ satisfies 
\begin{equation}\label{eq:n2}
\|ra_{i,J}\gamma_2\|\geq \frac{(\log N)^{B'}}{H^2}\text{ for every } 0<r\leq (\log N)^{B'},
\end{equation}
 then 
 $$
 \Big|\sum_{p\in [N,N+H]}e(a_{i,J}\gamma_2(p-N)^2+a_{i,I}\gamma_1(p-N))\log p\Big|\ll m(q)^{-16} H.
$$
So if every $0\neq |a_{i,J}|<m(q)$ satisfies \eqref{eq:n2}, then \eqref{eq:split} is, by summing the above, equal to 
$$
\lambda(J)\sum_{|a_{i,I}|<m(q)}\sum_{p\in [N,N+H]}c_{a_{i,I}}e_{a_{i,I}}(\gamma_1(p-N))\log p+{\rm O}(m(q)^{-14}H)=
$$$$
\lambda(J)\sum_{p\in [N,N+H]}f^i_I(\gamma_1(p-N))+{\rm O}\Big(q^{-9}\sum_{p\in [N,N+H]}\log p\Big).
$$
where in the last equality we used \eqref{eq:n1} and $m(q)>q$. Notice moreover, that since $2/3-\eta>7/12+1/20$, it follows by Lemma \ref{lem:shit} that the last term above is equal to ${\rm O}(q^{-9}H)$.
If we define $\theta(q)=\eta$, $H_0(q):=\max(e^{m(q)^{100}},H'(q))$ and $B(q)=B'+1$, then \eqref{eq:q2} implies that  \eqref{eq:n2} holds for every $0\neq |a_{i,J}|<m(q)$. Hence the above shows that
$$
\sum_{p\in [N,N+H]}f^{i}(\gamma_1(p-N),\gamma_2(p-N)^2)\log p=
\lambda(J)\sum_{p\in [N,N+H]}f^i_I(\gamma_1(p-N))\log p+{\rm O}(q^{-9}H).
$$
Note also that the above holds for {\bf every} $I\subset \T$. We will now use pigeonhole principle to show that there exists a collection of disjoint intervals $\{I_i\}_{i=1}^v$ of equal length $\in[\frac{1}{2q^2},\frac{2}{q^2}]$ such that 
$$
\Big|\sum_{p\in [N,N+H]}f^+_{I_i}(\gamma_1(p-N))\log p-\sum_{p\in [N,N+H]}f^-_{I_i}(\gamma_1(p-N))\log p\Big|={\rm O}(q^{-6}H).
$$
Notice that since $f^+_I\geq \chi_I\geq f^-_I$, the above statement immediately implies the proposition. Let $I=[a_i,b_i)$. Then the above difference is bounded above by
\begin{equation}\label{eq:aia}
\sum_{p\in [N,N+H]}\Big[\chi_{V_{q^{-9}}(a_i)}(\gamma_1(p-N))+\chi_{V_{q^{-9}}(b_i)}(\gamma_1(p-N))\Big]\log p.
\end{equation}
 For any $0\leq \ell\leq \frac{q^9}{2q^2}=q^7/2$, we now consider a collection of intervals 
 $$
 \Big\{\Big[\frac{j}{q^2}+\ell 2q^{-9},\frac{j}{q^2}+(\ell+1) 2q^{-9}\Big)\Big\}_{0\leq j\leq q^2-1}.
$$ 
  Notice that for every $\ell\neq \ell'$ any two such collections consist of pairwise disjoint intervals and the union over all $\ell\leq q^7/2$ covers $\T$. So by pigeonhole principle there exists $\ell_0\leq q^7/2$ such that (using also Lemma \ref{lem:shit} to count the number of primes in $[N,N+H]$)
\begin{multline}\label{eq:fun} 
\Big|\{p\in [N,N+H]\;:\; \gamma_1(p-N)\in \bigcup_{j\leq q^2-1}\Big[\frac{j}{q^2}+\ell_0 2q^{-9},\frac{j}{q^2}+(\ell_0+1) 2q^{-9}\Big)\}\Big|\leq\\ 2 q^2 2 q^{-9}H(\log H)^{-1}=4 q^{-7}H(\log H)^{-1}.
\end{multline}
Consider the midpoints $\{c_j\}_{j\leq q^2-1}$ of the intervals $\Big[\frac{j}{q^2}+\ell_0 2q^{-9},\frac{j}{q^2}+(\ell_0+1) 2q^{-9}\Big)$. They partition $\T$. Let now $\{I_j\}_{j\leq q^2-1}$ be the collection of intervals given by this partition, i.e. $I_j=[c_j,c_{j+1})$. Note that by definition, $|I_j|\in (\frac{1}{2q^2},\frac{2}{q^2})$. Moreover the intervals are pairwise disjoint and cover $\T$. It remains to notice that by \eqref{eq:fun}, \eqref{eq:aia} is bounded by
$$
\log N\cdot 4q^{-7}H(\log H)^{-1}={\rm O}(q^{-7}H),
$$
as $H\geq N^{2/3-\theta(q)}>N^{1/2}$.
 This finishes the proof.

\end{proof}
\subsection{Definition of the $G_\delta$ dense set.}
In this section we define the $G_{\delta}$ dense set of irrationals for which we will show Theorem \ref{thm:main2}. Roughly speaking, we will require that along a subsequence of denominators, we have that $q_{n+1}$ is much larger than $q_n$, so that we can apply  Proposition \ref{prop:cruc} with $N=q_{n+1}^{1/2}$ and $q=q_n$.

Let $\alpha\in \T$ and let $(a_i)_{i\in \N}$ denote the continued fraction expansion of $\alpha$. Let moreover $(q_i)_{i\in \N}$ be the sequence of denominators of $\alpha$, i.e. the sequence given by $q_0=q_1=1$ and 
$$
q_{n+1}=a_nq_n+q_{n-1}.
$$
We say that $\alpha \in \cD$ if there exists a subsequence $(n_k)$ such that for every $k\in \N$, we have 
\begin{equation}\label{eq:cD}
q_{n_k+1}\cdot [\log (q_{n_k+1})]^{-\frac{100}{\theta(q_{n_k})}}\geq \max\Big(e^{q_{n_k}}, \Big(H_0(q_{n_k})\Big)^2,q_{n_k}^{\frac{100}{\theta(q_{n_k})}}\Big),
\end{equation}
where $H_0(q_{n_k})$ and $\theta(q_{n_k})$ are given by Proposition \ref{prop:cruc}. The above condition expresses the fact that $q_{n_k+1}$ is large enough to guarantee that we can apply Proposition \ref{prop:cruc}. It follows (see eg. \cite{LKW}) that for any fixed 'rate' the set of $\alpha$ for which a subsequence of denominators grows with this rate is a $G_\delta$ dense set. We therefore have:
\begin{lemma}The set $\cD$ is a $G_\delta$ dense set.
\end{lemma}
\subsection{Ergodic sums estimates}
Let $\alpha\in \T$ with the sequence of denominators $(q_n)_{n\in \N}$. We first recall the following:

\begin{lemma}\label{lem:ase} Let $g\in {\rm BV}(\T)$. Then
$$
\sup_{x\in \T} \Big|S_{M}(g)(x)-M\int_\T g\,d{\rm Leb}\Big|=o(M).
$$
\end{lemma}
\begin{proof}The proof is classical and is a consequence of the
Denjoy-Koksma inequality:
$$
\Big|S_{q_n}(g)(x)-q_n\int_\T g\,d{\rm Leb}\Big|={\rm O}(Var(g));
$$
and the Ostrovski expansion, i.e. we write $M=\sum_{s\leq k} b_kq_k$, $b_k\leq \frac{q_{k+1}}{q_k}$ and use cocycle identity to write 
$$
|S_{M}(f)(x)-M\int_\T fd{\rm Leb}|={\rm O}({\rm Var}(g))\cdot\sum_{s\leq k}b_k.
$$
It remains to notice that 
$$
 \sum_{s\leq k}b_k=o(M).
$$
\end{proof}

 Assume now that $f\in C^2(\T\setminus\{0\})$ has power singularity with exponent $\gamma$. For $x\in \T$ and $n\in \N$ let $x_{n,min}:=\min_{0\leq i\leq n}\|x+i\alpha\|$.  We have  the following lemma (see Lemma 4.1 and Sublemma 1 in \cite{FFK}):
\begin{lemma}\label{lem:ta} There exists $C'>0$ such that for every  $n\in \N$ and $x\in \T$, we have 
\be\label{eq:f}
 \Big|S_{q_n}(f)(x)-q_n\int_\T f \,d\,{\rm Leb}\Big|\leq C'x_{q_n,min}^\gamma,
\ee
\be\label{eq:f'}
|S_{q_n}(f')(x)-f'(x_{q_n,min})|\leq C'q_n^{-1+\gamma}
\ee
and 
\be\label{eq:f''}
|S_{q_n}(f'')(x)-f''(x_{q_n,min})|\leq C'q^{-2+\gamma}_n.
\ee
\end{lemma}
\begin{proof} \eqref{eq:f} follows from Lemma 4.1. in \cite{FFK} and \eqref{eq:f'},\eqref{eq:f''} follow from Sublemma 1 in \cite{FFK}.
\end{proof}

 From the above lemma we get:

\begin{lemma}\label{lem:DK}There exists $C>0$, such that for $M\in [q_n,q_{n+1}]$, we have 
$$
\Big|S_M(f)(x)-M\int_\T fd {\rm Leb}\Big|\leq C\cdot 2^n q_n+C\cdot n\cdot\Big(\frac{1}{q_n^{1+\gamma}}M^{1+\gamma}q_{n+1}^{-\gamma} +x_{M,min}^\gamma\Big).
$$
\end{lemma}
\begin{proof} We argue by induction on $n$.For $n=1$ it is enough to take $C>10C'$ sufficiently large. By enlarging $C$ we can assume that $C>2^{-\gamma+1}C'$.
Assume the above holds for any $z\leq n$ and any $r=M\in [q_{z},q_{z+1}]$ and let $M\in [q_{n+1},q_{n+2}]$. Then  $M=kq_{n+1}+r$, $r< q_{n+1}$ and $k\leq \frac{M}{q_{n+1}}$. Let $r\in [q_{z},q_{z+1}]$, $z<n+1$. 
 Then 
$$
S_M(f)(x)=S_{kq_{n+1}}(f)(x)+S_{r}(f)(x+kq_{n+1}\alpha). 
$$

By definition, $(x+kq_{n+1}\alpha)_{r,min}\geq x_{M,min}$. Therefore and by the inductive assumption,
\be\label{eq:ada1}
|S_{r}(f)(x+kq_{n+1}\alpha)-r\int_\T f \,d\,{\rm Leb}|\leq C\cdot 2^zq_z+C\cdot z\cdot \Big(\frac{1}{q_z^{1+\gamma}}r^{1+\gamma}q_{z+1}^{-\gamma}+ x_{M,min}^\gamma\Big). 
\ee
Since $z<n+1$, $r<q_{z+1}$, we have 
\be\label{eq:ada2}
\frac{1}{q_z^{1+\gamma}}r^{1+\gamma}q_{z+1}^{-\gamma}<q_{n+1}
\ee
Moreover, by \eqref{eq:f} and the cocycle identity, 
$$
\Big|S_{kq_{n+1}}(f)(x)-kq_{n+1}\int_\T f \,d\,{\rm Leb}\Big|=
\Big|\sum_{i=0}^{k-1}(S_{q_{n+1}}(f)(x+iq_{n+1}\alpha)-q_{n+1}\int_\T f \,d\,{\rm Leb})\Big|\leq
$$$$
C'\sum_{i=0}^{k-1}(x+iq_{n+1}\alpha)^{\gamma}_{q_{n+1},min}.
$$
Notice that since $k\leq \frac{M}{q_{n+1}}\leq \frac{q_{n+2}}{q_{n+1}}$, the spacing between the points $\{x+iq_{n+1}\alpha\}_{i<k}$ is at least $\frac{1}{2q_{n+2}}$ and therefore the above sum is bounded above by
$$
C'x_{M,min}^\gamma+ C'(2q_{n+2})^{-\gamma}\sum_{i=0}^{k-1}i^{\gamma}\leq C'x_{M,min}^\gamma+2^{-\gamma+1}C'k^{1+\gamma}q_{n+2}^{-\gamma}.
$$
Then using that $k\leq \frac{M}{q_{n+1}}$ and \eqref{eq:ada1}, \eqref{eq:ada2}, we get 
$$
\Big|S_M(f)(x)-M\int_\T f \,d\,{\rm Leb}\Big|\leq $$$$
C\cdot 2^nq_n +C\cdot n\cdot( q_{n+1}+x_{M,min}^\gamma)+C'x_{M,min}^\gamma+2^{-\gamma+1}C'\Big(\frac{M}{q_{n+1}}\Big)^{1+\gamma}q_{n+2}^{-\gamma}\leq 
$$$$
C2^{n+1}q_{n+1}+C\cdot (n+1)\cdot \Big(\Big(\frac{M}{q_{n+1}}\Big)^{1+\gamma}q_{n+2}^{-\gamma}+x_{M,min}^\gamma\Big)
$$
This finishes the proof.
\end{proof}
\begin{remark}\label{rem:DK}Notice that if $M\leq q_{n+1}$, then
$$
\Big|S_M(f)(x)-M\int_\T fd {\rm Leb}\Big|\leq C\cdot 2^n q_n+C\cdot n\cdot\Big(\frac{1}{q_n^{1+\gamma}}M^{1+\gamma}q_{n+1}^{-\gamma} +x_{M,min}^\gamma\Big). 
$$
Indeed, if $M\in [q_n,q_{n+1}]$ then it is immediate from the above lemma. If $M\in [q_z,q_{z+1}]$ with $z<n$, then we use the above lemma for $z$ and 
$$
\frac{1}{q_z^{1+\gamma}}r^{1+\gamma}q_{z+1}^{-\gamma}<q_{n}.
$$
\end{remark}

\begin{lemma}\label{lem:appr} Fix $n\in \N$ and let $2\leq k\leq q_{n+1}^{3/4}/q_n$ be such that 
\begin{equation}\label{eq:nsing}
\{x+i\alpha\}_{i<kq_n}\cap \Big[-\frac{1}{L},\frac{1}{L}\Big]=\emptyset, 
\end{equation}
for some $L<q_{n+1}/4$. Then
$$
S_{kq_n}(f)(x)=kS_{q_n}(f)(x)+\Big(kS_{q_n}(f)(x)\Big)^2\Big[\frac{S_{q_n}(f')(x)(q_n\alpha)}{(S_{q_n}(f)(x))^2}\Big]+ {\rm O}\Big(\frac{L^3q_n^3 k^3}{q_{n+1}^2}+\frac{kL^2q_n^2}{q_{n+1}}\Big).
$$
\end{lemma}
\begin{proof}Assume WLOG that $q_n\alpha<0$. Notice that by cocycle identity
$$
S_{kq_n}(f)(x)-kS_{q_n}(f)x=\sum_{\ell=0}^{k-1}\Big(S_{q_n}(f)(x+\ell q_n\alpha)-S_{q_n}(f)(x)\Big).
$$
The spacing between the points $\{x+\ell q_n\alpha\}_{\ell <k}$ is $\|q_n\alpha\|<\frac{1}{q_{n+1}}$. So using \eqref{eq:nsing} and $L<q_{n+1}/4$, it follows that for every $\ell <k$,  $0\notin [x+\ell q_n\alpha,x]$. Therefore, 
 the sum above is for some $\theta_\ell\in [x+\ell q_n\alpha, x]$, $\ell<k$, equal to 
$$
\sum_{\ell=0}^{k-1}\Big(S_{q_n}(f')(x)(\ell q_n\alpha)+S_{q_n}(f'')(\theta_\ell)\|\ell q_n\alpha\|^2\Big).
$$
Notice that by \eqref{eq:f''} in Lemma \ref{lem:ta} and by \eqref{eq:nsing} it follows that for every $\ell<k$, 
$$
|S_{q_n}(f'')(\theta_\ell)|\|\ell q_n\alpha\|^2\ll L^3q_n^3 \cdot \frac{k^2}{q_{n+1}^2}.
$$
Therefore, and by \eqref{eq:f'} in Lemma \ref{lem:ta} (bounding the derivative by $\frac{L^2q_n^2}{q_{n+1}}$),
$$
S_{kq_n}(f)(x)-kS_{q_n}(f)x=\Big[\sum_{\ell<k} \ell\Big]S_{q_n}(f'(x))(q_n\alpha)+ {\rm O}\Big( \frac{L^3q_n^3 k^3}{q_{n+1}^2}  \Big)=
$$$$
k^2S_{q_n}(f'(x))(q_n\alpha)+ {\rm O}\Big( \frac{L^3q_n^3 k^3}{q_{n+1}^2}  \Big)+{\rm O}\Big((k^2-\sum_{\ell<k}k)\cdot \frac{L^2q_n^2}{q_{n+1}}\Big)=
$$$$
k^2S_{q_n}(f'(x))(q_n\alpha)+ {\rm O}\Big( \frac{L^3q_n^3 k^3}{q_{n+1}^2}  \Big)+{\rm O}\Big(\frac{k L^2q_n^2}{q_{n+1}}\Big).
$$
This finishes the proof.
\end{proof}
\begin{remark}
Notice that the error term in the above lemma being small implies that (at least)  $k<q_{n+1}^{2/3}$. This is the main reason why we need Proposition \ref{prop:cruc} for intervals $[N,N+H]$ with $H<N^{2/3-\epsilon}$.
\end{remark}
Finally, we have the following lemma:
\begin{lemma}\label{lem:smallder} For $n\in \N$ let $q_{n+1}\geq e^{q_n}$ and let $x_n\in \T$ be any point such that $S_{q_n}(f')(x_n)=0$. Then 
$$
\{x\in \T\;:\; |S_{q_n}(f')(x)|<q_{n+1}^{-1/10}\}\subset \bigcup_{i< q_n} [-2q_{n+1}^{-1/10}+x_n+i\alpha,2q_{n+1}^{-1/10}+x_n+i\alpha]
$$
\end{lemma}
\begin{proof}Consider the partition of $\T$ given by points $\{-i\alpha\}_{i<q_n}$ and let $I=[a,b)$ be any interval in this partition. Then $S_{q_n}(f')(\cdot)$ is differentiable on $I$: by defintion, $0\notin I+i\alpha$ for $i<q_n$. Moreover, $S_{q_n}(f')(\cdot)$ is monotone on $I$ (see \eqref{eq:f''}) and $\lim_{x\to a^+}S_{q_n}(f')(x)=-\infty=-\lim_{x\to b^-}S_{q_n}(f')(x)$. Hence there exists a unique point $x_I\in I$ such that $S_{q_n}(f')(x_I)=0$. We will show that 
\be\label{mina}
\min(\|x_I-a\|,\|x_I-b\|)\gg q_n^{-1}.
\ee
Indeed,  note that by \eqref{eq:f'},
$$
0=|S_{q_n}(f')(x_I)|\geq |f'((x_I)_{q_n,min})|-Cq_n^{-1+\gamma}.
$$
Moreover, $(x_I)_{q_n,min}=\min(\|x_I-a\|,\|x_I-b\|)$ and so 
(see \eqref{eq:asu}), 
$$
8q_n^{-1+\gamma}\geq |f'((x_I)_{q_n,min})| \gg \min(\|x_I-a\|,\|x_I-b\|)^{-1+\gamma}
$$
This gives \eqref{mina}. So it follows that for any $\theta\in [-\frac{1}{q_{n+1}}+x_I,x_I+\frac{1}{q_{n+1}}]$,
\be\label{eq:wa}
|S_j(f'')(\theta)|\ll q_n^{2-\gamma}\;\; \text{ for every } j<q_n.
\ee
Indeed, by \eqref{eq:asu} it follows that $f''=g''+f_+''$, where $g''$ is bounded and $f_+''>0$. Then by \eqref{eq:f''} (for $f_+''$) and the above
$$
|S_j(f'')(\theta)|\leq S_j(f''_+)(\theta)+{\rm O}(j)\ll S_{q_n}(f''_+)(\theta)+{\rm O}(j)\ll q_n^{2-\gamma}.
$$
 Note that for any $x\in I$, and some $\theta_x\in [x,x_I]$,
$$
|S_{q_n}(f')(x)|=|S_{q_n}(f')(x)-S_{q_n}(f')(x_I)|=S_{q_n}(f'')(\theta_I)|x-x_I|.
$$
We have $x_{q_n,min}\leq \frac{2}{q_n}$ (the $q_n$- orbit of every point is $\frac{1}{q_n}$ dense) and therefore by the two above (and \eqref{eq:f''}), 
$$
|S_{q_n}(f')(x)|\geq C  \Big(\frac{2}{q_n}\Big)^{-2+\gamma}|x-x_I|\geq |x-x_I|.
$$
From this it follows that for any $I$,
\be\label{eq:ja}
\{x\in \T\;:\; |S_{q_n}(f')(x)|<q_{n+1}^{-1/10}\}\cap I\subset [-q_{n+1}^{-1/10}+x_I,x_I+q_{n+1}^{-1/10}].
\ee
It remains therefore to show that for every $J$ in the partition, there exists $j<q_n$, such that 
\be\label{eq:z1}
[-q_{n+1}^{-1/10}+x_J,x_J+q_{n+1}^{-1/10}]\subset [-2q_{n+1}^{-1/10}+x_I+j\alpha,x_I+j\alpha+2q_{n+1}^{-1/10}].
\ee
Let $j<q_n$ be unique such that $x_I+j\alpha\in J$ (existence of $j$ follows from \eqref{mina} and $\|q_n\alpha\|\leq \frac{2}{q_{n+1}}\leq \frac{2}{e^{q_n}}$). Then by \eqref{eq:wa} and $q_{n+1}>e^{q_n}$, for some $\theta\in [x_I,x_I+q_n\alpha]$,
$$
|S_{q_n}(f')(x_I+j\alpha)|=|S_{q_n}(f')(x_I+j\alpha)-S_{q_n}(f')(x_I)|=$$$$
|S_{j}(f')(x_I+q_n\alpha)-S_j(f')(x_I)|\leq |S_j(f'')(\theta)|\frac{1}{q_{n+1}}\leq q_{n+1}^{-1/20}.
$$
So by \eqref{eq:ja} for $J$,
$$
x_I+j\alpha\in  [-q_{n+1}^{-1/10}+x_J,x_J+q_{n+1}^{-1/10}].
$$
This immediately gives \eqref{eq:z1}.
\end{proof}

\subsection{Proof of Proposition \ref{prop:forback}}
In this section we assume that a function $\bar{\psi}$ satisfying i.-iii. is fixed. For simplicity of notations we denote it by $\psi$.  Let $(q_n)$ denote the sequence of denominators of $\alpha\in \cD$. Let $(n_k)$ be the subsequence constructed in \eqref{eq:cD}. The only property of $(n_k)$ that we use in this section is that $q_{n_k+1}\geq e^{q_{n_k}}$.

We start with the following lemma:
\begin{lemma}\label{lem:ad}There exists $c>0$ such that for every $z\in \{+,-\}$ every $t\in [\frac{q_{n_k+1}}{\log n_k}, cq_{n_k+1}]$ and every $x\in \T$ for which $(x,s)\in \T^f$ satisfies
\be\label{eq:tz}
\{T_{z\cdot w}(x,s)\}_{w\leq t} \cap \{(y,r)\in \T^f\;:\; \|y\|<\frac{1}{4}q_{n_k+1}^{-1}\}=\emptyset,
\ee
we have
\be\label{eq:min}
\max \Big(|N(x,s,z\cdot t)-z\cdot t|,|z\cdot t-S_{N(x,s,z\cdot t)}(f)(x)|\Big)=o(t).
\ee
\end{lemma}
\begin{proof} Let $c:=\frac{\inf_\T f}{16}$. Assume WLOG that \eqref{eq:tz} holds with $z=+$ and $t\in [\frac{q_{n_k+1}}{\log n_k}, cq_{n_k+1}]$. If  it is the case for $z=-$, we proceed analogously. 
Note that for fixed $(x,s)$ and sufficiently large $t$, $(\inf_\T f)N(x,s,t)\leq S_{N(x,s,t)}(f)(x)\leq t+s\leq 2t \leq 2cq_{n_k+1}$ and so $N(x,s,t)\leq q_{n_k+1}/8$. Therefore by \eqref{eq:f} and \eqref{eq:tz},
$$
|t-S_{N(x,s,t}(f)(x)|\leq |s|+ f(x+N(x,s,t)\alpha)={\rm O}(q_{n_k+1}^{1+\gamma})=o(t).
$$
Moreover, by \eqref{eq:tz}, it follows that $x_{N(x,s,t),min}\geq  \frac{1}{4q_{n_k+1}}$) and since $N(x,s,t)\leq q_{n_k+1}/8$, by Lemma \ref{lem:DK} and Remark \ref{rem:DK}, (recall that $-1<\gamma<0$),
$$
|S_{N(x,s,t)}(f)(x)-N(x,s,t)|\leq 
$$$$
C2^{n_k}q_{n_k}+C\cdot(n_k+1)\cdot \Big(\Big(\frac{N(x,s,t)}{q_{n_k}}\Big)^{1+\gamma}q_{n_k+1}^{-\gamma}+x_{N(x,s,t),min}^\gamma\Big)=o(t),
$$
by the bound on $t$ in the statement of the lemma since as shown above, \\
$(\inf_\T f)N(x,s,t)\leq 2t$ (and moreover, $q_{n_k+1}\geq e^{q_{n_k}}$). The two above inequalities finish the proof.
\end{proof}

We have one more simple lemma:
\begin{lemma}\label{lem:sas} There exists $c>0$ such that for every $(x,s)\in \T^f$ and every sufficiently large $k$ (depending on $(x,s)$),
$$
 \Big\{w: |w|<cq_{n_k+1}, T_{w}(x,s)\in \{(y,r)\in \T^f\;:\; \|y\|<\frac{1}{4}q_{n_k+1}^{-1}\}\Big\}
$$

is an interval\footnote{Might be empty.} which is a subset of one of $[-cq_{n_k+1},0)$ or $(0,cq_{n_k+1}]$.

\end{lemma}
\begin{proof} Let $c:=\frac{\inf_\T f}{16}$. Recall that the first coordinate of  $T_{w}(x,s)$ is equal to $x+N(x,s,w)\alpha$. Moreover, 
$$
(\inf_\T f)|N(x,s,w)|\leq |S_{N(x,s,w)}(f)(x)|\leq| w|+s\leq s+cq_{n_k+1}\leq 2cq_{n_k+1},
$$
($s$ is fixed and $n_k\to +\infty$). Therefore, 
\be\label{eq:nxsw}
|N(x,s,w)|<q_{n_k+1}/8
\ee
Note that every connected component of 
$$\{T_{w}(x,s)\}_{|w|\leq cq_{n_k+1}} \cap \{(y,r)\in \T^f\;:\; \|y\|<\frac{1}{4}q_{n_k+1}^{-1}\}$$
 starts by visiting $[-\frac{1}{4}q_{n_k+1}^{-1},\frac{1}{4}q_{n_k+1}^{-1}]$. Note that for any $i_1,i_2\in [-\frac{q_{n_k+1}}{4},\frac{q_{n_k+1}}{4}]$, 
$$
\|(x+i_1\alpha)-(x+i_2\alpha)\|=\|(i_1-i_2)\alpha\|> \frac{1}{2q_{n_k+1}}
$$
since $|i_1-i_2|\leq q_{n_k+1}/2<q_{n_k+1}$. Therefore and by the bound on $N(x,s,w)$ (see \eqref{eq:nxsw}) the above orbit visits the set $[-\frac{1}{4q_{n_k+1}},\frac{1}{4q_{n_k+1}}]$ at most once. Hence indeed the intersection is one interval (might be empty).


\end{proof}
\begin{proposition}\label{prop:fos} If \eqref{eq:tz} holds for $t=T$ and $z\in\{+,-\}$, then \eqref{eq:forback} holds for $T$ (with the same $z$).
\end{proposition}
Before we prove the above proposition, let us see how it immediately implies Proposition \ref{prop:forback}.
\begin{proof}[Proof of Proposition \ref{prop:forback}]
 Let $c$ be smaller from the two constants in Lemma \ref{lem:ad} and Lemma \ref{lem:sas}. Fix $t\in [\frac{q_{n_k+1}}{\log n_k}, cq_{n_k+1}]$. By Lemma \ref{lem:sas}, one of the semi-orbits $\{T_{w}(x,s)\}_{w\leq cq_{n_k+1}}$ or $\{T_{-w}(x,s)\}_{w\leq cq_{n_k+1}}$ is disjoint with  $\{(y,r)\in \T^f\;:\; \|y\|<\frac{1}{4}q_{n_k+1}^{-1}\}$. So \eqref{eq:tz} holds either for $z=+$ or $-$. Then statement then follows by Proposition \ref{prop:fos}.
\end{proof}

\begin{proof}[Proof of Proposition \ref{prop:fos}]
We can WLOG assume that $\int_{\T^f}\psi\;d\,{\rm Leb}^f=0$. 
$$
\int_0^T\psi(T^\alpha_t(x,s))dt=
\int_0^T\psi(T^\alpha_t(x,0))dt+{\rm O}(s),
$$
and since $s$ is fixed, it is enough to estimate the integral on the RHS. Note that 
\begin{equation}\label{eq:inter}
\int_0^T\psi(T^\alpha_t(x,0))dt=\int_0^{S_{N(x,0,T)}(f)(x)}\psi(T^\alpha_t(x,0))dt+\int_{S_{N(x,0,T)}(f)(x)}^T\psi(T^\alpha_t(x,0))dt.
\end{equation}

By Lemma \ref{lem:ad}, $0\leq T-S_{N(x,0,T)}(f)(x)\leq f(x+N(x,0,T)\alpha)\ll f(\frac{1}{q_{n_k+1}})=o(T)$ and hence the second summand is negligible.

For $\delta>0$, let $X_\delta:=\{(x,s)\in \T^f\;:\; \|x\|>\delta\}$ and let $F(x,\delta):=\{t\leq S_{N(x,0,T)}(f)(x) \;:\; T^\alpha_t(x,0)\in X_\delta\}$. We claim that for every $\epsilon>0$ there exists $\delta>0$ and $T_\delta>0$ such that for $T\geq T_\delta$,
\be\label{eq:fdel}
|F(x,\delta)|\geq(1-\epsilon)S_{N(x,0,T)}(f)(x).
\ee
Indeed, consider the function $\bar{f}(x)=\chi_{(\delta,1-\delta)}\cdot f(x)$. Then notice that 
$$
|F(x,\delta)|\geq \sum_{i=0}^{N(x,0,T)}\bar{f}(x+i\alpha)
$$
Then by Lemma \ref{lem:ase},
$$
\sum_{i=0}^{N(x,s,T)}\bar{f}(x+i\alpha)=N(x,s,T)\int_{\T} \bar{f} d {\rm Leb}+o\Big(N(x,s,T)\Big).
$$
Note that for every $\epsilon>0$ there exists $\delta$ such that  $\int_{\T} \bar{f} d {\rm Leb}\geq 1-\epsilon/4$. Moreover, by Lemma \ref{lem:ad},
$$N(x,s,T)\geq T-o(T)\geq (1-\epsilon/4)S_{N(x,s,T)}(f)(x).$$ 

In particular it follows by \eqref{eq:fdel} that it is enough to estimate the integral 
$$
\int_0^{S_{N(x,s,T)}(f)(x)}\psi(T^\alpha_t(x,s))dt
$$
 restricted to the set $F(x,\delta)$.
 Let $Z(x):=\int_{0}^{f(x)}\psi(T_t(x,0))dt$. Then the integral $
\int_0^{S_{N(x,0,T)}(f)(x)}\psi(T^\alpha_t(x,0))dt$ restricted to the set $F(x,\delta)$ us equal to 
$$
\sum_{i=0}^{N(x,0,T)-1}Z\cdot \chi_{[\delta,1-\delta]}(x+i\alpha).
$$
It now remains to notice that the function $Z\cdot \chi_{[\delta,1-\delta]}$ is of bounded variation, and so by Lemma \ref{lem:ase}, 
$$
\sum_{i=0}^{N(x,0,T)-1}Z\cdot \chi_{[\delta,1-\delta]}(x+i\alpha)=o(N(x,0,T)),
$$
and $N(x,0,T)\leq 2T$ by Lemma \ref{lem:ad}. This finishes the proof.
\end{proof}

\subsection{Proof of Theorem \ref{thm:main2}}
In this section we assume that $\alpha \in \cD$ and $f$ has power singularity with exponent $\gamma$ (see \eqref{eq:asu}). Let $(q_{n_k})$ be the sequence coming from the fact that $\alpha\in \cD$. 
In this section we assume that the function $\bar{\psi}$ satisfying i.-iii. is fixed. For simplicity of notations we denote it by $\psi$.

To simplify notation we drop $\alpha$ and $\gamma$ from the notation for $(T_t^{\alpha,\gamma})$. We start with the following lemma:
\begin{lemma}\label{lem:navi} Let $\delta>0$ be such that $-\gamma(1+\delta)<1$. Then for
$$ 
x\notin\bigcup_{i=0}^{q_{n_k}-1}\Big(\Big[-\frac{1}{q_{n_k}^{1+\delta}}, \frac{1}{q_{n_k}^{1+\delta}}\Big]-i\alpha\Big), 
$$
and any $\ell<2q_{n_k}$, we have
$$
\int_{0}^{S_{q_{n_k}}(f)(x)}\psi(T_{t+\ell}(x,0))=o(S_{q_{n_k}}(f)(x)).
$$
\end{lemma}
\begin{proof}Let $T_{\ell}(x,0)=(x+n_\ell \alpha,s_\ell)$. Then 
$$
|S_{q_{n_k}}(f)(x)-S_{q_{n_k}}(f)(x+n_\ell\alpha)|=|S_{n_\ell}(f)(x)-
S_{n_\ell}(f)(x+q_{n_k}\alpha)|.
$$
By the assumption on $x$ and since $q_{n_k+1}\geq e^{q_{n_k}}$,
$$
|S_{n_\ell}(f)(x)-S_{n_\ell}(f)(x+q_{n_k}\alpha)|=\frac{1}{q_{n_k+1}}|S_{n_\ell}(f')(\theta)|
$$
for some $\theta\in [x,x+q_{n_k}\alpha]$, where (since $\|q_{n_k}\alpha\|<e^{-q_{n_k}}$),
$$
\theta\notin\bigcup_{i=0}^{q_{n_k}-1}\Big(\Big[-\frac{1}{2q_{n_k}^{1+\delta}}, \frac{1}{2q_{n_k}^{1+\delta}}\Big]-i\alpha\Big).
$$
Since $n_\ell\leq 2q_{n_k}$, $|S_{n_\ell}(f')(\theta)|\ll q_{n_k}f(\theta_{q_{n_k},min})\leq q_{n_k}^4$ (the last inequality by the restriction on $\theta$ above). Therefore, 
$$
|S_{q_{n_k}}(f)(x)-S_{q_{n_k}}(f)(x+n_\ell\alpha)|=o(1), 
$$
and so it is enough to estimate the integral
$$ 
\int_{0}^{S_{q_{n_k}}(f)(x+n_\ell\alpha )}\psi(T_{t}(x+n_\ell \alpha,s_\alpha)).
$$
This now follows the same  steps as estimating the first integral in \eqref{eq:inter}.
\end{proof}

We restate Theorem \ref{thm:main2} to shorten the notation:
\begin{proposition}\label{prop:pri} There exists $c>0$ such that for $N_k\in [\frac{q_{n_k+1}}{\log k}, cq_{n_k+1}]$, for every $z\in\{+,-\}$ and $m\geq 0$,
\begin{equation}\label{eq:nk}
\Big|\sum_{p\leq N_k}\psi(T_{z\cdot (p-m)}(x,s))\log p-\int_{0}^{N_k}\psi(T_{z\cdot t}(x,s))dt\Big|=o(N_k),
\end{equation}
for every $(x,s)\in \T^f$.
\end{proposition}
\begin{proof}Assume WLOG that $\int_{\T^f}\psi \,d\,{\rm Leb}= 0$. We will show \eqref{eq:nk} for $z=+$, the proof in case $z=-$ is symmetric and follows analogous lines. Note that 
$$
\sum_{p\leq N_k}\psi(T_{p-m}(x,s))\log p=\sum_{p-m\in [0,N_k]}\psi(T_{p-m}(x,s))\log p+{\rm O}(m\log N_k),
$$
and since $m\geq 0$ is fixed, it is enough to estimate the first sum.
Let $\delta>0$ be such that $-\gamma(1+\delta)<1$ (such $\delta$ exists since $\gamma\in (-1,0)$). Let
\begin{equation}\label{eq:az}
I_a:=\bigcup_{i=0}^{q_{n_k}-1}\Big(\Big[-\frac{1}{q_{n_k}^{1+\delta}}, \frac{1}{q_{n_k}^{1+\delta}}\Big]-i\alpha\Big),
\end{equation}
and $I_b:=\T\setminus I_a$. Let $I_a^f:=\{(y,s)\in \T^f\;:\; y\in I_a \}$ and analogously we define $I_b^f$.

 We split the interval $[0,N_k]$ into two disjoint subsets:
$$
A:=\{t\in [0,N_k]\;:\; T_{t}(x,s)\in I_a^f\}
$$
and 
$$
B:=\{t\in [0,N_k]: T_{t}(x,s)\in I_b^f\}.
$$
We assume WLOG that $q_{n_k}\alpha<0$ the reasoning in the other case is analogous.  Let $t_0\in [0,N_k]$ be the smallest for which $T_{t_0}(x,s)\in I_a^f$ and let $t_1\in [t_0,N_k]$ be the smallest for which $T_{t_1}(x,s)\notin I_a^f$. Let 
$$A_0:=\Big\{t\in [0,N_k]\;:\; T_{t}(x,s)\in \Big[- \frac{1}{4q_{n_k+1}},\frac{1}{4q_{n_k+1}} \Big]^f, s\geq k \Big\}.$$ We have the following:
\\
\\
\textbf{CLAIM:} \begin{itemize}
\item[P1.] $A=[t_0,t_1)$ and $B=[0,t_0)\cup [t_1,N_k]$;
\item[P2.] $A_0$ is an interval;
\item[P3.]$A\setminus A_0$ is a union of at most two intervals and $|A\setminus A_0|=o(N_{k})$.
\end{itemize}
Before we prove the {\bf CLAIM} let us show how it implies the proposition.

We then naturally split the integrals in \eqref{eq:nk} into integrals over $A_0$,  $A\setminus A_0$ and $B$. Notice that by P3. and Remark \ref{rem:prim} ($m$ is fixed),
$$
|\sum_{p-m\in A\setminus A_0}\psi(T_{p-m}(x,s))\log p|\ll  |\sum_{p-m\in A\setminus A_0}\log p|=o(N_k)
$$
and also $|\int_{A\setminus A_0}\psi(T_{t}(x,s))dt|=o(N_k)$.
 Therefore it remains to estimate the terms in \eqref{eq:nk} over $A_0$ and $B$. By definition of $A_0$ and iii. in the definition of $\psi$, for every $t\in A_0$ for which the vertical coordinate of $T_{t}(x,s)$ is $\geq \log N_k$, $\psi(T_t(x,s))=\psi_\infty+o(1)$. Moreover, since $A_0$ is an interval, the measure  of $t\in A_0$ (cardinality of $p-m\in A_0$) for which  the vertical coordinate of $T_{t}(x,s)$ (of $T_{p-m}(x,s)$) is $\leq \log N_k$ is bounded above by ${\rm O}(\log N_k)$. Hence 
$$
\sum_{p-m\in A_0}\psi(T_{p-m}(x,s))\log p=(\psi_\infty+o(1)\Big)\sum_{p-m\in A_0}\log p+{\rm O}(\log^2 N_k)
$$
and 
$$
\int_{A_0}\psi(T_{t}^{\alpha,\gamma}(x,s))dt=(\psi_\infty+o(1)\Big)|A_0|+{\rm O}(\log N_k).
$$
If $|A_0|\leq N_k^{9/10}$, then the above sum over primes above is $o(N_k)$ and the integral is also of order $o(N_k)$. On the other hand if the interval $A_0$ satisfies $A_0\geq N_{k}^{9/10}$, then by Lemma \ref{lem:shit}, the above sum over primes is equal to $\psi_{\infty}|A_0|+o(N_k)$. Therefore, 
$$
\Big|\sum_{p-m\in A_0}\psi(T_{p}(x,s))\log p-
\int_{A_0}\psi(T_{t}(x,s))dt\Big|=o(N_k),
$$
and hence it remains to estimate \eqref{eq:nk} over the set $B=[0,t_0)\cup [t_1,N_k]$. If the length of $[0,t_0)$ or $[t_1,N_k]$ is less than $\frac{N_k}{\log \log k}$, then analogously to the above reasoning, both the sum and the integral are of order $o(N_k)$ and hence such interval is negligible. Therefore it remains to prove the following: let $I\in \{[0,t_0),[t_1,N_k]\}$ be such that $|I|\geq \frac{N_k}{\log \log k}$, then
\begin{equation}\label{eq:10}
\Big|\sum_{p-m\in I}\psi(T_{p-m}(x,s))\log p-\int_{I}\psi(T_{t}(x,s))dt\Big|=o(|I|).
\end{equation}
We will argue with $I=[t_1,N_k)$, the estimates for $I=[0,t_0]$ are analogous (and slightly less technical since we start at $0$). By definition, $T_{t_1}(x,s)=(\tilde{x},0)$ (it is the first time we are outside $I_a^f$). Let $L:=q_{n_k+1}^{2/3-\theta(q_{n_k})}$. Since $k$ is fixed in what follows, we denote $\theta=\theta(q_{n_k})$.

 For $0\leq u\leq \frac{N_k}{L}$, consider the intervals  
 $$
 W_u:=[t_1+S_{uLq_{n_k}}(f)(\tilde{x}),t_1+ S_{(u+1)Lq_{n_k}}(f)(\tilde{x})).
 $$
  Notice that by cocycle identity, 
\begin{equation}\label{eq:jk}
|W_u|=S_{Lq_{n_k}}(f)(\tilde{x}+uLq_{n_k}\alpha)\geq (\inf_\T f) Lq_{n_k}\geq q_{n_k+1}^{2/3-\theta}.
\end{equation}
Let $S$ me maximal such that $\bigcup_{u<S}W_s\subset I$. Since for  $t\in I$, $T_t(x,s)\in I_b^f$, it follows that for every $u< S$, 
$$(\tilde{x}+uLq_{n_k}\alpha)_{Lq_{n_k},min}\notin \Big[-\frac{1} {q_{n_k}^{1+\delta}}, \frac{1} {q_{n_k}^{1+\delta}}\Big].
$$ 
This by \eqref{eq:f} and splitting into sums of length $q_{n_k}$ implies that for $u<S$, we have 
\be\label{eq:wssm}
|W_u|\leq S_{Lq_{n_k}}(f)(\tilde{x}+uLq_{n_k}\alpha)\leq Lq_{n_k}+  C\cdot Lq_{n_k}^{-\gamma(1+\delta)} \leq q_{n_k+1}^{2/3-\theta/2},
\ee
since by \eqref{eq:cD}, $Lq_n=q_{n+1}^{2/3-\theta}q_n\leq \frac{1}{2}q_{n+1}^{2/3-\theta}$ and $Lq_{n_k}^{-\gamma(1+\delta)}\leq q_{n_k+1}^{2/3-\theta}q_{n_k}^{-\gamma(1+\delta)}\leq  \frac{1}{2}q_{n_k+1}^{2/3-\theta/2}$.
 Therefore, 
$$
\Big|I\setminus\Big(\bigcup_{u<S}W_u \Big)\Big|<q_{n_k+1}^{2/3},
$$
and hence the contribution of the interval $I\setminus\Big(\bigcup_{u<S}W_u \Big)$ is negligible (i.e. $o(|I|)$) in \eqref{eq:10}, see Remark \ref{rem:prim}). Let $W_u=[a_u,b_u)$. Then, by the definition of $W_u$, $T_{a_u}(x,s)=(\tilde{x}+uLq_{n_k}\alpha,0)$.

Note that since $q_{n_k}\alpha<0$, the points $\{\tilde{x}+uLq_{n_k}\alpha\}_{u\leq S}$ satisfy $\tilde{x}>\tilde{x}+Lq_{n_k}\alpha>\ldots>\tilde{x}+SLq_{n_k}\alpha$ and the spacing between them is  $Lq_{n_k}\alpha$ (which is of order $q_{n_k+1}^{-1/3-\theta}$). For $u\leq S$, we will estimate 
\begin{equation}\label{eq:nn}
\sum_{p-m\in [a_u,b_u)}\psi(T_{p-m-a_u}(T_{a_u}x,s))\log p.
\end{equation}

Let  $z_k$ be any point such that $S_{q_{n_k}}(f')(z_k)=0$ and define: 
$$
Z_k:= \bigcup_{i<q_{n_k}} [-2q_{n_k+1}^{-1/10}+z_k+i\alpha,2q_{n_k+1}^{-1/10}+z_k+i\alpha]
$$
(this is the set where the derivative is small, see Lemma \ref{lem:smallder}). Let $u\in S'$ iff $T_{a_u}(x,s)\in Z_k$ or $T_{b_u}(x,s)\in Z_k$. By the above remark on the order and spacing between the points $\{\tilde{x}+uLq_{n_k}\alpha\}_{u< S}$ it follows that 
$$
|S'|\ll q_{n_k+1}^{1/3+\theta} q_{n_k+1}^{-1/10}.
$$
Therefore, by \eqref{eq:wssm} and Remark \ref{rem:prim} it follows that 
$$
\sum_{u\in S'}\sum_{p-m\in [a_u,b_u)}\psi(T_{p-m-a_u}(T_{a_u}x,s))\log p\ll \sum_{u\in S'}\sum_{p-m\in [a_u,b_u)}\log p\ll$$$$
q_{n_k+1}^{1/3+\theta} q_{n_k+1}^{-1/10}q_{n_k+1}^{2/3}=o(N_k),
$$
the last inequality since $N_k\geq \frac{q_{n_k+1}}{\log k}$.  Similarly, $$\sum_{u\in S' }\int_{[a_u,b_u)}\psi(T_t(x,s))\,d\,{\rm Leb}=o(N_k).
$$
 Therefore it is enough to estimate \eqref{eq:nn} for $u\in [0,S)\setminus S'$. Let $x_u:=\tilde{x}+uLq_{n_k}\alpha$. For such $u$ by the definition of $Z_k$ and Lemma \ref{lem:smallder} we know that 
\begin{equation}\label{eq:lder}
|S_{q_{n_k}}(f')(x_u)|\geq q_{n_k+1}^{-1/10}.
\ee
 Let 
\be\label{eq:zl}
\gamma_1:=\Big[S_{q_{n_k}}(f)(x_u)\Big]^{-1}.
\ee
Note that by definition, $(x_u,0)=T_{a_u}(x,s)\in I_b^f$. Therefore, in particular (see the definition of $I_a$) it follows that $(x_u)_{q_{n_k},min}\geq q_{n_k}^{-1-\delta}$. So by \eqref{eq:f},
\be\label{eq:sol}
S_{q_{n_k}}(f)(x_u)\leq q_{n_k}+Cq_{n_k}^{\gamma(-1-\delta)}\in [\frac{q_{n_k}}{2},2q_{n_k}],
\ee
since we have chosen $\delta$ to satisfy $-\gamma(1+\delta)<1$. Moreover, by \eqref{eq:f'},
\be\label{eq:sol2}
|S_{q_{n_k}}(f')(x_u)|\leq  q_{n_k}^3.
\ee
 By \eqref{eq:cD} and \eqref{eq:jk} it follows that $H:=|b_u-a_u|=|W_s|\geq H_0(q_{n_k})$ and moreover, 
$$
a_u^{2/3-\theta}\leq N^{2/3-\theta}_k\leq q^{2/3-\theta
}_{n_{k+1}}\leq H. 
$$
So the assumptions of Proposition \ref{prop:cruc} are satisfied with $N=a_u$ and $H=b_u-a_u$. Let $\{I_i\}_{i=1}^v$ be the disjoint collection of intervals coming from Proposition \ref{prop:cruc} for $q=q_{n_k}$, $N=a_u$ and $H=b_u-a_u$ (and assume that $c_i$ is the midpoint of $I_i$). Fix $i\leq v$ and take all $p-m\in [a_u,b_u)$ such that 
\be\label{eq:ga}
\gamma_1(p-m-a_u) \mod 1\in I_i.
\ee
Since $|I_i|\leq \frac{2}{q^2}$ and by \eqref{eq:sol}, we get that there exists $M_{i,p}\in \N$ such that 
\be\label{eq:pak}
\Big|(p-m-a_u)-M_{i,p}S_{q_{n_k}}(f)(x_u)-\tilde{c_i}\Big|\ll q_{n_k}^{-1},
\ee
where $\tilde{c_i}:=S_{q_{n_k}}(f)(x_u)\cdot c_i$. Since $c_i$ is the midpoint of $I_i$ and all the intervals $\{I_i\}_{i=1}^v$ have equal length $\in [\frac{1}{2q^2_{n_k}},\frac{2}{q^2_{n_k}}]$, it follows that $\{\tilde{c_i}\}_{i=1}^v$ are equispaced $\in [0,S_{q_{n_k}}(f)(x_u)]$ and
 by \eqref{eq:sol},
 \be\label{eq:sol3}
\tilde{c_i}=i\xi_v\;\; \text{ with }\;\; 0<\xi_v\ll q_{n_k}^{-1}.
 \ee

By \eqref{eq:pak}, \eqref{eq:wssm} and \eqref{eq:sol},
$$
M_{i,p}\leq \frac{2(b_u-a_u)+S_{q_{n_k}}(f)(x_u)}{S_{q_{n_k}}(f)(x_u)}\ll {\frac{|W_u|}{q_{n_k}}+1\ll \frac{q_{n_k+1}^{2/3-\theta/2}}{q_{n_k}}}.
$$
If $j\leq  q_{n_k+1}^{2/3}$, $j=k_jq_{n_k}+r_j$ with $0\leq r_j<q_{n_k}$, then
$$
\|(x_u+j\alpha)-(x_u+r_j\alpha)\|\leq j\|q_{n_k}\alpha\|\leq q_{n_k+1}^{-1/3}.
$$
 Since $(x_u,0)\in I_b^f$ it follows that $x_u\notin I_a$ (see \eqref{eq:az}) and therefore

\begin{equation}\label{eq:asdd}
\{x_u+j\alpha\}_{j<q_{n_k+1}^{2/3}}\cap \Big[-\frac{1}{2}q_{n_k}^{-1-\delta},\frac{1}{2}q_{n_k}^{-1-\delta}\Big]=\emptyset.
\end{equation}
Let 
$$
\bar{\gamma}:=\Big[\frac{S_{q_{n_k}}(f')(x_u)(q_n\alpha)}{(S_{q_{n_k}}(f)(x_u))^2}\Big].
$$
Then by \eqref{eq:sol2},
\be\label{eq:barg}
|\bar{\gamma}|\leq \frac{q_{n_k}^3}{q_{n_k+1}}\;\; \Big(\leq q_{n_k}^{-6}\Big)
\ee
 By Lemma \ref{lem:appr}  with $L=\frac{1}{2}q_{n_k}^{1+\delta}$ (see \eqref{eq:asdd}), we get 
$$
M_{i,p}S_{q_{n_k}}(f)(x_u)=S_{M_{i,p}q_{n_k}}(f)(x_u)- \Big(M_{i,p}S_{q_{n_k}}(f)(x_u)\Big)^2\bar{\gamma}+
$$$$
{\rm O}\Big(L^3q_{n_k}^3M_{i,p}^3q_{n_k+1}^{-2}+M_{i,p}L^2q_{n_k}^2q_{n_k+1}^{-1}\Big).
$$
Note that by the bound on $M_{i,p}$ and by \eqref{eq:cD}, we get 
$$
L^3q_{n_k}^3M_{i,p}^3q_{n_k+1}^{-2}\ll q_{n_k}^{10}q_{n_k+1}^{-3\theta}\ll q_{n_k}^{-2}
$$
and
$$
M_{i,p}L^2q_{n_k}^2q_{n_k+1}^{-1}\ll q_{n_k}^6q_{n_k+1}^{-1/3}\leq 
q_{n_k}^{-2}.
$$

Therefore, by \eqref{eq:pak} and using $\bar{\gamma}\leq q_{n_k}^{-6}$, we get  
$$
\Big|(p-m-a_s)-\tilde{c_i}- S_{M_{i,p}q_{n_k}}(f)(x_u)- (p-m-a_s)^2\bar{\gamma}\Big|\ll q_{n_k}^{-1}.
$$
Let now $\{J_h\}_{h=1}^w$ be a disjoint collection of intervals of equal  length $\in [\frac{1}{2q^2_{n_k}},\frac{2}{q_{n_k}^2}]$ covering $\T$ (with midpoints $d_h$). Define
$\gamma_2:=\frac{\bar{\gamma}}{S_{q_{n_k}}(f)(x_u)}$. Let 
 $h$ and $p$ (which we already assume satisfies \eqref{eq:ga}) satisfy
\be\label{eq:g2}
(p-m-a_u)^2\gamma_2\mod 1\in J_h.
\ee

Using \eqref{eq:sol} implies that, for some $R_{h,i,p}\in \N$
$$
|(p-m-a_u)^2\bar{\gamma}-\tilde{d_i}-R_{h,i,p}S_{q_{n_k}}(f)(x_u)|\ll q_{n_k}^{-1},
$$
where (analogously to $\{\tilde{c_i}\}$) $\tilde{d_h}:=S_{q_{n_k}}(f)(x_u)\cdot d_h$. Since $d_h$ is the midpoint of $J_h$ and all the intervals $\{J_h\}_{h=1}^w$ have equal length $\in [\frac{1}{2q^2_{n_k}},\frac{2}{q^2_{n_k}}]$, it follows that $\{\tilde{d_h}\}_{h=1}^w$ are equispaced $\in [0,S_{q_{n_k}}(f)(x_u)]$ and
 by \eqref{eq:sol},
 \be\label{eq:sol4}
d_h=h\zeta_w\;\; \text{ with }\;\; 0<\zeta_w\ll q_{n_k}^{-1}.
 \ee
Moreover, by the definition of $\bar{\gamma}$ (see also \eqref{eq:barg}) and \eqref{eq:wssm} and \eqref{eq:cD}
$$
|R_{h,i,p}|\ll q_{n_k}^{-1}+ \frac{\tilde{d_i}}{S_{q_{n_k}}(f)(x_u)}+ \frac{(b_u-a_u)^2q_{n_k}^3}{q_{n_k+1}}\ll 2+ q_{n_k+1}^{1/3-\theta}q_{n_k}^3\leq q_{n_k+1}^{1/3-\theta/2}.
$$
Therefore, for $p$ satisfying \eqref{eq:ga} and \eqref{eq:g2}, 
$$
\Big|(p-m-a_u)-\tilde{c_i}-\tilde{d_h}- S_{M_{i,p}q_{n_k}}(f)(x_u)-R_{h,i,p}S_{q_{n_k}}(f)(x_u)\Big|\ll q_{n_k}^{-1}.
$$
By the bound on $R_{h,i,p}$ and $M_{i,p}$, \eqref{eq:asdd} and \eqref{eq:f'},

$$
\Big|R_{h,i,p}S_{q_{n_k}}(f)(x_u)-S_{R_{h,i,p}q_{n_k}}(f)(x_u+M_{i,p}q_{n_k}\alpha)\Big|\leq$$$$
 \sum_{u=0}^{R_{h,i,p}-1}|S_{q_{n_k}}(f)(x_u)-S_{q_{n_k}}(f)(x_u+(M_{i,p}+u)q_{n_k}\alpha)|\leq 
$$$$
R_{h,i,p} q_{n_k}^3\frac{M_{i,p}+R_{i,p}}{q_{n_k+1}}\ll q_{n_k+1}^{-\theta}q_{n_k}^3\ll q_{n_k}^{-1},
$$
the last inequality by \eqref{eq:cD}.
Therefore and by cocycle identity it follows that for $p$ satisfying \eqref{eq:ga} and \eqref{eq:g2},
$$
\Big|(p-m-a_u)-\tilde{c_i}-\tilde{d_h}- S_{(R_{h,i,p}+M_{i,p})q_{n_k}}(f)(x_u)\Big|\ll q_{n_k}^{-1}
$$
Hence for such $p$, 
 $$
 T_{p-m-a_u}(T_{a_u}x,s)=T_{p-m-a_u}(x_u,0)=T_{\tilde{c_i}+\tilde{d_h}+{\rm O}(q_{n_k}^{-1})}\Big(x_u+(R_{h,i,p}+M_{i,p})q_{n_k}\alpha,0\Big)
$$
Since $R_{h,i,p}+M_{i,p}<q_{n_k+1}^{3/4}$, 
$$
\|(R_{h,i,p}+M_{i,p})q_{n_k}\alpha\|\ll q_{n_k+1}^{1/4}\leq q_{n_k}^{-3},
$$
the last inequality by \eqref{eq:cD}.
Therefore, by the continuity of $\psi$ (see conditions i. and ii. in the definition of $\bar{\psi}$) for $p$ satisfying \eqref{eq:ga} and \eqref{eq:g2},
$$
|\psi(T_{p-m-a_u}(T_{a_u}x,s))-\psi(T_{\tilde{c_i}+\tilde{d_h}}(x_u,0))|=o(1).
$$
Hence and by  \eqref{eq:ga} and \eqref{eq:g2}, 
$$
\sum_{p-m\in [a_u,b_u)}\psi(T_{p-m-a_u}(T_{a_u}x,s))\log p=$$$$
\sum_{i=1}^v\sum_{h=1}^w\psi(T_{\tilde{c_i}+\tilde{d_h}}(x_u,0))\Big[\sum_{p-m\in[a_u,b_u)}\chi_{I_i\times J_h}(\gamma_1(p-m-a_u),\gamma_2(p-m-a_u)^2)\Big]+ $$$$
o(\sum_{p-m\in [a_u,b_u)}\log p).
$$
 By Lemma \ref{lem:shit} the last term is $o(|W_u|)$ and hence can be neglected.

Notice that by \eqref{eq:sol}, \eqref{eq:sol2} and \eqref{eq:lder}, 
$$
|\gamma_2|=\Big|\frac{\bar{\gamma}}{S_{q_{n_k}}(f)(x_u)}\Big|=\Big|\frac{S_{q_{n_k}}(f')(x_u)\|q_{n_k}\alpha\|}{(S_{q_{n_k}}(f)(x_u))^3}\Big|\in [ 4q_{n_k}^{-3}q_{n_k+1}^{-1-1/10},2q_{n_k+1}^{-1}] 
$$
and so, for every $0\neq r\leq \log^B N_k$ (recall that $N_k\leq q_{n_k+1}$)
$$
\|r\gamma_2\|=r|\gamma_2|\geq \frac{4\log^Bq_{n_k+1}}{q_{n_k}^3q_{n_k+1}^{1+1/10}}\geq \frac{1}{|W_u|^2},
$$
the last inequality by the bound on $|W_u|$ (see \eqref{eq:jk}).
So by Proposition \ref{prop:cruc} (recall that $m$ is fixed), 
$$
\sum_{p-m\in[a_u,b_u)}\chi_{I_i\times J_h}(\gamma_1(p-m-a_u),\gamma_2(p-m-a_u)^2)=$$$$
\lambda(J_h)\sum_{p-m\in[a_u,b_u)}\chi_{I_i}(\gamma_1(p-m-a_u))+
{\rm O}(|W_u|q_{n_k}^{-6})
$$
Let $B_i:=\sum_{p-m\in[a_u,b_u)}\chi_{I_i}(\gamma_1(p-m-a_u))\log p$.

Then (recall that $|I_i|,|J_h|\geq \frac{1}{2q^2_{n_k}}$, and so $v\cdot w\leq 4q_{n_k}^2$)
$$
\sum_{p-m\in [a_u,b_u)}\psi(T_{p-m-a_u}(T_{a_u}x,s))\log p=\lambda(J_h)\sum_{i=1}^vB_i\sum_{h=1}^w\psi(T_{\tilde{c_i}+\tilde{d_h}}(x_u,0))+ {\rm O}(|W_u|q_{n_k}^{-2}).
$$
Note that by \eqref{eq:sol4}, $|\tilde{d_{h+1}}-\tilde{d_h}|=\zeta_w\leq q_{n_k}^{-1}$. Since $\psi$ is continuous along the flow direction (see i. and ii.), we have 
$$
\Big|\lambda(J_h)\sum_{h=1}^w\psi(T_{\tilde{d_h}+\tilde{c_i}}(x_u,0))-\frac{1}{\tilde{d_w}}\int_{0}^{\tilde{d_w}}\psi(T_{t+\tilde{c_i}}(x_u,0)dt\Big|=o(1).
$$
Moreover, 
$$
\Big|\frac{1}{\tilde{d_w}}\int_{0}^{\tilde{d_w}}\psi(T_{t+\tilde{c_i}}(x_u,0)dt\Big|=o(1).
$$
Indeed, notice that $\tilde{d_w}\in[S_{q_{n_k}}(f)(x_u)-q_{n_k}^{-1/2}, S_{q_{n_k}}(f)(x_u)]$. Therefore, up to an error of order $o(1)$, the above integral is equal to 
$$
\Big|\frac{1}{S_{q_{n_k}}(f)(x_u)}\int_{0}^{S_{q_{n_k}}(f)(x_u)}\psi(T_{t+\tilde{c_i}}(x_u,0)dt\Big|.
$$
The statement then follows by Lemma \ref{lem:navi} and \eqref{eq:sol} (recall that $x_u\in I_b$).
Therefore and by Lemma \ref{lem:shit},
$$
 \lambda(J_1)\sum_{i=1}^vB_i\sum_{h=1}^w\psi(T_{\tilde{c_i}+\tilde{d_h}}(x_u,0))=o(\sum_{i=1}^vB_i)=o(\sum_{p-m\in[a_u,b_u)}\log p )= o(|W_u|).
$$
This shows that \eqref{eq:nn} is $o(|W_u|)=o(b_u-a_u)$. Summing over $u$, we get 
$$
|\sum_{p-m\in I}\psi(T_{p-m}(x,s))\log p|=o(|I|).
$$
Note that since $I\subset B$, it follows that 
$$
\{T_{w}(x,s)\}_{w\in I}\subset I_b^f
$$
equivalently, for $w\in [0,|I|]$,
\be\label{eq>zd2}
T_w(T_{u_1}(x,s))\notin I_a^f
\ee
where $I=[u_1,u_2]$. In particular, let $T_{u_1}(x,s)=(\tilde{x},\tilde{s})$. Then \eqref{eq>zd2} implies that \eqref{eq:tz} is satisfied with $z=+$. So by Proposition \ref{prop:fos}, 
$$
\Big|\int_{I}\psi(T_{t}(x,s))dt\Big|=
\Big|\int_{0}^{|I|}\psi(T_{t}(\tilde{x},\tilde{s})dt\Big|=o(|I|).
$$
 This finishes the proof of \eqref{eq:10} and hence also the proof of the main theorem. So it remains to prove the {\bf CLAIM}.
\end{proof}
\begin{proof}[Proof of the {\bf CLAIM}]
Assume WLOG that $q_{n_k}\alpha<0$. 
We first prove $P1$.  Note that $T_{t}(x,s)\in I_a^f$ is equivalent to 
$x+N(x,s,t)\alpha\in I_a$. The function $N(x,s,\cdot)$ is a locally constant jump function (with jump size $1$). Let $t_0$ be the first such that $x+N(x,s,t_0)\alpha\in I_a$. By the definition of $I_a$ (it is a tower) and $t_0$, this means $x+N(x,s,t_0)\alpha\in \Big[-\frac{1}{q_{n_k}^{1+\delta}}, \frac{1}{q_{n_k}^{1+\delta}}\Big]-(q_{n_k}-1)\alpha$. Then, since $I_a$ is a tower, for every $j\leq q_{n_k}-1$, $x+N(x,s,t_0)\alpha+j\alpha\in I_a$. If $x+N(x,s,t_0)\alpha+rq_{n_k}\alpha\in I_a$, equivalently $x+N(x,s,t_0)\alpha+rq_{n_k}\alpha\in \Big[-\frac{1}{q_{n_k}^{1+\delta}}, \frac{1}{q_{n_k}^{1+\delta}}\Big]-(q_{n_k}-1)\alpha$, then for every $j\leq q_{n_k}-1$, $x+N(x,s,t_0)\alpha+(rq_{n_k}+j)\alpha\in I_a$.

Notice that each return to $\Big[-\frac{1}{q_{n_k}^{1+\delta}}, \frac{1}{q_{n_k}^{1+\delta}}\Big]-(q_{n_k}-1)\alpha$ shifts the point by $q_{n_k}\alpha$. Since $S_{N(x,s,N_k)}\leq (\inf_\T f)N_k\leq q_{n_k+1}/3$, it follows that after leaving $\Big[-\frac{1}{q_{n_k}^{1+\delta}}, \frac{1}{q_{n_k}^{1+\delta}}\Big]-(q_{n_k}-1)\alpha$, we will not (by shifting over $q_{n_k}\alpha$) return to it before time $S_{N(x,s,N_k)}(f)(x)$. Hence indeed the set $A$ is an interval. Analogously we show that $B$ is an interval. This finishes the proof of P1. 

For P2., we analogously notice that $T_t(x,s)\in [-\frac{1}{4q_{n_k+1}},\frac{1}{4q_{n_k+1}}]^f$, implies that $x+N(x,s,t)\alpha\in [-\frac{1}{4q_{n_k+1}},\frac{1}{4q_{n_k+1}}]$. Note that there is at most one $r< q_{n_k+1}/2$ such that $x+r\alpha\in [-\frac{1}{4q_{n_k+1}},\frac{1}{4q_{n_k+1}}]$. Indeed:
$$
\|(x+r\alpha)-(x+r'\alpha)\|\geq \|q_{n_k}\alpha\|\geq \frac{1}{q_{n_k+1}}.
$$
 Since $N(x,s,N_k)\leq q_{n_k+1}/3$, it follows that there is at most one value $N\leq q_{n_k+1}/3$ for which $x+N\alpha\in [-\frac{1}{4q_{n_k+1}},\frac{1}{4q_{n_k+1}}]$. This finishes the proof of P2.
 
Note that the first part of P3. is an immediate consequence of P1. and P2. So it remains to show that $|A\setminus A_0|=o(N_k)$. Let 
$$
\chi_k(x)=\chi_{I_a}(x)-\chi_{[-\frac{1}{4q_{n_k+1}},\frac{1}{4q_{n_k+1}}]}(x),
$$
and let $f_k(x):=\chi_k(x)\cdot f(x)$.
By the definition of the special flow and $A$ and $A_0$ one needs to show that 
$$
|S_{N(x,s,N_k)}(f_k)(x)|=o(N_k).
$$
Note that $(\inf_\T f)N(x,s,N_k)\leq S_{N(x,s,N_k)}(f)(x)\leq N_k+s\leq cq_{n_k+1}+s\leq (\inf_\T f)q_{n_k+1}$ (by decreasing $c$ if necessary).
Therefore, by Remark \ref{rem:DK}\footnote{We remark that Lemma \ref{lem:DK} and Remark \ref{rem:DK} hold also for the function $f_k$ instead of $f$. The proofs of them for $f_k$ are the same as the corresponding proofs for $f$.}
$$
|S_{N(x,s,N_k)}(f_k)(x)|\leq C2^nq_n +C\cdot n\cdot \frac{1}{q_{n_k}^{1+\gamma}}q_{n_k+1}+C\cdot n \cdot q_{n_k+1}^{-\gamma}=o(N_k),
$$
since $N_k\geq \frac{q_{n_k+1}}{\log n_k}$. This finishes the proof of the {\bf CLAIM}.

\end{proof}


\begin{thebibliography}{9}
\bibitem{Ar}V.\ Arnol'd, {\em Topological and ergodic properties of closed 1-forms with incommensurable periods},
Funktsionalnyi Analiz i Prilozheniya, 25, no. 2 (1991), 1--12. (Translated in: Functional
Analysis and its Applications, 25, no. 2, 1991, 81--90).
\bibitem{Bom}E.\ Bombieri,{\em On the large sieve}, Mathematika 12 (1965), 201--225.
\bibitem{Bour} J.\ Bourgain, {\em On the correlation of the M\"obius function with rank-one systems} J. Anal. Math., 120, (2013), 105--130.
\bibitem{BO} J.\ Bourgain, {\em An approach to pointwise ergodic theorems}, In Geometric aspects of functional analysis
(1986/87), volume 1317 of Lecture Notes in Math., 204--223, Springer, Berlin, 1988.
\bibitem{CW}J.\ Chaika, A.\ Wright, {\em A smooth mixing flow on a surface with non-degenerate fixed points},
J. Amer. Math. Soc. 32 (2019), 81--117.
\bibitem{COFOSI} I.\ P.\ Cornfeld, S.\ V.\ Fomin, Ya.\ G.\ Sinai,{\em Ergodic theory}, Grundlehren der Math. Wissenschaften 245 (1982) 486, Springer, New York.
\bibitem{FAY}B.\ Fayad, {\em Weak mixing for reparameterized linear flows on the torus}, Ergodic Theory Dynam. Systems, volume 22 (2002), no. 1, p. 187--201.
\bibitem{fay} B.\ Fayad, {\em Polynomial decay of correlations for a class of smooth flows on the two torus}, Bull.
SMF 129 (2001), 487--503.
\bibitem{FFK}B.\ Fayad, G.\ Forni, A.\ Kanigowski, {\em Lebesgue spectrum of countable multiplicity for conservative flows on the torus}, arXiv:1609.03757. 
\bibitem{FK} B.\ Fayad, A.\ Kanigowski, {\em On multiple mixing for a class of conservative surface flows},
Inv. Math. 203 (2) (2016), 555--614.
\bibitem{FM}S.\ Ferenczi, C.\ Mauduit, {\em On Sarnak's conjecture and Veech's question for interval exchanges}, J.
Anal. Math., 134, (2018), 545--573.
\bibitem{FFT}L.\ Flaminio, G.\ Forni, J.\ Tanis, {\em Effective equidistribution of twisted horocycle flows and horocycle maps}, Geom. Funct. Anal, (2016) 26, 1359--1448.
\bibitem{Gr}B.\ Green, {\em On (not) computing the M\"obius function using bounded depth circuits}, Combin. Probab.Comput., 21,(2012), 942--951.
\bibitem{GT}B.\ Green, T.\ Tao, {\em The M\"obius function is strongly orthogonal to nilsequences}, Ann. of Math. (2),
175(2) (2012), 541--566.
\bibitem{KLR}A.\ Kanigowski, M.\ Lema{\'n}czyk, M.\ Radziwi{\l}{\l}, {\em Prime number theorem for analytic skew-products}, submitted, arXiv:2004.01125 
\bibitem{KATOK} A.\ B.\ Katok, {\em Spectral properties of dynamical systems with an integral invariant on the torus}. Dokl. Akad. Nauk SSSR 223 (1975), 789--792.
\bibitem{Kat44}A.\ Katok, {\em Combinatorial constructions in ergodic theory and dynamics} volume 30 of University Lecture
series. American Mathematical Society, 2003.
\bibitem{KS}K.\ M.\ Khanin, Ya.\ G.\ Sinai, {\em Mixing for some classes of special flows over rotations of
the circle}, Funktsionalnyi Analiz i Ego Prilozheniya, 26, no. 3 (1992), 1--21 (Translated in:
Functional Analysis and its Applications, 26, no. 3, 1992, 155--169).
\bibitem{Koc1}A.\ V.\ Kochergin, {\em Mixing in special flows over a shifting of segments and in smooth flows on
surfaces}, Mat. Sb. (N.S.) , 96 (138) (1975), 471--502.
\bibitem{Koc2}A.\ V.\ Kochergin, {\em Nonsingular saddle points and the absence of mixing}, Mat. Zametki, 19 (3)
(1976), 453--468 (Translated in: Math. Notes, 19:3: 277-286.)
53
\bibitem{Koc3}A.\ V.\ Kochergin, {\em Nondegenerate fixed points and mixing in flows on a two-dimensional
torus I}: Sb. Math. 194 (2003) 1195--1224; II: Sb. Math. 195 (2004) 317--346.
\bibitem{KOL}A.\ N.\ Kolmogorov, {\em On dynamical systems with an integral invariant on the torus}, Doklady Akad. Nauk SSSR 93 (1953), 763--766.
\bibitem{Koukoupulos} D. Koukoulopoulos, {\em Primes in short arithmetic progressions}, Int. J. Number Theory, 11(5) 2015, 1499--1521.
\bibitem{LKW}J.\ Kwiatkowski,\ M. Lema{\'n}czyk, D.\ Rudolph,{\em A class of real cocycles having an analytic coboundary modification}, Israel J. Math., 87 (1994),337--360.
\bibitem{MR}C.\ Mauduit, J.\ Rivat, {\em Prime numbers along Rudin-Shapiro sequences}, J. Eur. Math. Soc. (JEMS),
17(10), (2015), 2595--2642.
\bibitem{Mat-Shao} K.\ Matomäki, X.\ Shao, {\em Discorrelation between primes in short intervals and polynomial
phases}, arxiv:1902.04708, 2019.
\bibitem{MU} C.\ M\"ullner, {\em Automatic sequences fulfill the Sarnak conjecture} Duke Math. J., 166(17) (2017), 3219--3290.
\bibitem{Nov}S.\ P.\ Novikov, {\em The Hamiltonian formalism and a multivalued analogue of Morse theory},
Uspekhi Mat. Nauk 37 (1982), no. 5 (227), 3--49.
\bibitem{OPP}A.\ Perelli, J.\ Pintz, S.\ Salerno, {\em Bombieri's theorem in short intervals. II.} Invent. Math., 79(1) (1985), 1--9.
\bibitem{Rav}D.\ Ravotti, {\em Quantitative mixing for locally Hamiltonian flows with saddle loops on compact
surfaces}, Ann. Henri Poincaré 18 (12) (2017), 3815--3861.
\bibitem{SU}P.\ Sarnak, A.\ Ubis, {\em The horocycle flow at prime times}, J. Math. Pures Appl. (9), 103(2) (2015), 575--618.
\bibitem{Shk} M.\ D.\ Shklover, {\em On dynamical systems on the torus with continuous spectrum}, Izv. Vuzov 10 (1967), 113--124.
\bibitem{Sh}N.\ A.\ Shah, {\em Limit distributions of polynomial trajectories on homogeneous spaces}, Duke Math.
J. 75 (3), (1994), 711--732.
\bibitem{Ulc1}C.\ Ulcigrai, {\em Mixing of asymmetric logarithmic suspension flows over interval exchange
transformations}, Ergod. Th. Dyn. Sys. 27 (2007), 991--1035.
\bibitem{Ulc2}C.\ Ulcigrai, {\em Absence of mixing in area-preserving flows on surfaces}, Ann. of Math. 173
(2011), 1743--1778.
\bibitem{VI}I.\ M.\ Vinogradov, {\em The method of trigonometrical sums in the theory of numbers}, Trav. Inst. Math.
\bibitem{Vi2}A.\ I.\ Vinogradov,{\em The density hypothesis for Dirichlet L-series}, Izv. Akad. Nauk SSSR Ser. Mat.
29 (1965), 903--934.
\bibitem{Ven}] A. Venkatesh, {\em Sparse equidistribution problems, period bounds and subconvexity}. Ann. of Math.
172, (2010), 989--1094.
\bibitem{WI}M.\ Wierdl, {\em Pointwise ergodic theorem along the prime numbers}, Israel J. Math., 64(3), (1989), 315--336.
\end{thebibliography}
\end{document}